 \documentclass[final,1p,times]{elsarticle}
\usepackage{algorithmic}
\usepackage{algorithm}
\usepackage{amsmath}
\usepackage{amsfonts}
\usepackage{latexsym}
\usepackage{epsfig}
\usepackage{graphicx}
\usepackage{longtable}
\usepackage{rotating}

\usepackage{amssymb}
 \usepackage{amsthm}
 \newtheorem{proposition}{Proposition}

\journal{to be submitted in extended form}

\begin{document}

\begin{frontmatter}



\title{Dynamic Clustering of Histogram Data Based on Adaptive Squared Wasserstein Distances}


\author[add1]{Antonio Irpino}
 \author[add1]{Rosanna Verde}
 \author[add2]{Francisco de A.T. De Carvalho}

\address[add1]{Dipartimento di Studi Europei e Mediterranei \\
Second University of Naples \\
81100 Caserta, Italy\\
               \{antonio.irpino, rosanna.verde\}@unina2.it}
\address[add2]{Centro de Informatica--CIn/UFPE, \\
       Av. Prof. Luiz Freire, s/n, Ciadade Universitaria, CEP 50.740--540 Recife-PE, Brazil\\
       fact@cin.ufpe.br}

\begin{abstract}
This paper deals with clustering methods based on adaptive distances for histogram data using a dynamic clustering algorithm. Histogram data describes individuals in terms of empirical distributions. These kind of data can be considered as complex descriptions of phenomena observed on complex objects: images, groups of individuals, spatial or temporal variant data, results of queries, environmental data, and so on.  The Wasserstein distance is used to compare two histograms. The Wasserstein distance between histograms is constituted by two components: the first based on the means, and the second, to internal dispersions (standard deviation, skewness, kurtosis, and so on) of the histograms.\\
To cluster sets of histogram data, we propose to use Dynamic Clustering Algorithm, (based on adaptive squared Wasserstein distances) that is a k-means-like algorithm for clustering a set of individuals into $K$ classes that are apriori fixed.
 The main aim of this research is to provide a tool for clustering histograms, emphasizing the different contributions of the histogram variables, and their components, to the definition of the clusters. We demonstrate that this can be achieved using  adaptive distances.\\
Two kind of adaptive distances are considered: the first takes into account the variability of each component of each descriptor for the whole set of individuals; the second takes into account the variability of each component of each descriptor in each cluster.
We furnish interpretative tools of the obtained partition based on an extension of the classical measures (indexes) to the use of  adaptive distances in the clustering criterion function.
Applications on synthetic and real-world data corroborate the proposed procedure.
\end{abstract}


\end{frontmatter}



\section{Introduction}
\label{intro}
In many real experiences, data are collected and/or represented by histograms representing empirical distributions of phenomenon. In the framework of computer vision, the characteristics of images are usually represented as histograms of different masses. Other fields of applications also use histogram descriptions: for privacy preserving matters, data about a phenomenon (for example, flows of a bank account) can be summarized by histograms, as well as the dissemination of official statistics.
Cluster analysis aims to collect a set of objects in a
number of homogeneous clusters according to the values they assume
with respect to a set of observed variables.

In this paper we deal with a clustering procedure to partition a set
of histogram data, in a predefined
number of clusters. Histogram data were introduced in the context of
Symbolic Data Analysis by \citet{BoDid00} and they are
defined by a set of contiguous intervals of real domain which
represent the support of each histogram, with associated a system of
weights (frequencies, densities).

Symbolic Data Analysis (SDA) is a domain in the area of knowledge
discovery related to multivariate analysis, pattern recognition and
artificial intelligence, aiming to provide suitable methods
(clustering, factorial techniques, decision trees, etc.) for
managing aggregated data described by multi-valued variables, i.e.,
where the cells of the data table contain sets of categories,
intervals, or weight (probability) distributions (for further
insights about the SDA approach, see \cite{BoDid00},
\cite{Billard2007} and \cite{DiNoi08}).

Several proposals have been presented in the literature for
clustering histogram data (see \cite{IrVer06},\cite{IRVERLEC06},
\cite{VeIR08}, \cite{VerIRP08did}). Dynamic Clustering (DC)
(\cite{Did71},\cite{DiSIM76}) is proposed as a suitable method to
partition a set of data represented by frequency distributions. We
recall that DC needs to define a proximity function, to assign the
individuals to the clusters, and to choose a suitable way to represent the
clusters by an element which optimizes a criterion
function. Further, the representative element of a cluster, called
prototype, has to be consistent with the description of the clustered
elements: i.e., if data to be clustered are distributions, the prototype must be
also a distribution.

The DC method \cite{Did71} is a general partitioning algorithm of a set of
objects in $K$ clusters. It looks for the solution by optimizing a
criterion of best fitting between the partition and the
representation of the clusters of such partition. In DC, the choice of a suitable dissimilarity plays a central role for the definition of the allocation and of representation phases. \emph{k-means} algorithm is a particular case of DC when in the criterion function is used the squared Euclidean distance. According to the nature of data and the chosen dissimilarity function, DC is a more general schema of partition around a set of prototypes. In the case of \emph{k-means} algorithm prototypes are the means of each cluster. According to the optimized criterion, prototypes can also be regression lines, factorial axis, etc. \\
The comparison of histogram data can be seen a particular case of comparison of distribution functions. Several distances and dissimilarities have been presented in the literature, some of these, used for histogram data, are presented in \cite{VerIRP08did}. Another good review of distances between distributions can be found in \cite{GiSu02}. In the special field of computing vision, \cite{Rubn00} introduced the Earth Mover's distance (EMD) for color and texture images. This distance can be applied to distributions of points. It is worth of notice that EMD for histograms of pixel intensities is equivalent to the Mallows, or Wasserstein, distance on probability distributions \cite{Mall72}, \cite{LeBick01}.
 The family of distances based on Wasserstein metric permits to obtain interesting interpretative results based on the characteristics or on the moments of the compared distributions (see \cite{VerIRP08did} for details).
It has also proved such distance can be decomposed into two
components: the first related to the means of the histograms, and
the second to their internal dispersions.

One of the main issues in clustering (in multivariate analysis, in
general) is to take into account the roles of the
different variables. While the use of standard distances allows to find spherical groups (like in k-means), the main advantage of using adaptive distances is the possibility of
identifying clusters of different size (in terms of variability) and
orientation in the space (in terms of main alignment of the cluster
with one ore more directions of a set of variables).
One way is to homogenize the variables by means
of a standardization step. Another way is to use an adaptive distance in the
clustering algorithm that includes, in the optimization process, the
tuning of a set of weights to associate with each variable (for all
the clusters or within each cluster). \citet{DECAYVES09,DECLEC09,DECSOUZ07}, \cite{DeCDeS10} proposed several
adaptive distances for the dynamic clustering of intervals and histogram data. In the  paper \cite{DeCDeS10} is introduced a system of weights for each variable for Euclidean distance. The weights are dependent from the distance but not the optimization process so that the same schema can be  be easily extended for the Wasserstein distance too.
\\

In the present paper we propose two
adaptive approaches for clustering histogram data, based on
Wasserstein distance. This metric allows to compare histogram data
with respect both frequency and support while Euclidean distance compare histogram taking into account only one component (frequency or support). Further, \citet{IrpinoR07} showed that is possible to decompose the Wasserstein distance between two histograms in two (additive and independent) components: the first is related to the locations of the histograms, while the second is related to the different variability of the two histograms. In order to take advantage from this decomposition, we propose the following approaches for the definition of adaptive distances.

In the first approach, we propose to associate two sets of weights for each variable and each component in which is decomposed the distance, the first is globally estimated for all the clusters at once; the second is locally estimated for each cluster.

In order to furnish clustering interpretative tools, we propose an extension of classical ratios
based on within (intra-cluster) sum of squares, the between (inter-cluster) sum of squares and the total sum of squares
having proved the decomposition of the inertia of a set of histogram data computing with the adaptive (squared) Wasserstein distances.

This paper is organized as follows: in Section \ref{sec:histo}, we
introduce the definitions of histogram data and of the Wasserstein
distance for histograms. In Section \ref{sec:DCA}, starting from
the Dynamic Clustering Algorithm with non-adaptive distances, we
propose two schemas where the
adequacy criterion is based on adaptive squared Wasserstein
distances for histogram data. In section \ref{sec:tools}, we introduce some
tools for the interpretation of the clustering results. In Section
\ref{sec:ExperimentalResults},  two applications are shown: one using synthetic data in order to prove the
usefulness of the proposed methods based on the variability
structure of the data; the other one, using a real dataset in order to demonstrate
the application in a real situation and to show how to interpret the
results of a classic clustering task on histogram data. Section
\ref{sec:conclu} ends the paper with some conclusions and
perspectives about the proposed clustering methods.

\section{Histogram data and Wasserstein distance}
\label{sec:histo} The clustering of data expressed as histograms can be useful to discover
typologies of phenomena on the basis of the similarity of their
distributions. In general, clustering techniques depend from the choice of a suitable dissimilarity, where the adjective \emph{suitable} is related to the capability of the dissimilarity to take into account the nature of the data and of their representation space.
In this section, we give a definition of the histogram data and we propose the Wasserstein distance as a suitable metric for comparing them.

\subsection{Histogram Data}
Histogram is a cheap way for the representation of aggregate data or empirical distributions. Indeed, if any model of distribution can be assumed, the distribution of a set of data can be represented by histograms: i.e. a set of contiguous intervals (bins), of equal or of different width, associated with a set of weights (empirical frequencies or densities).

Formally, let $Y$  be a continuous variable defined on a finite support
$S=[min(y);Max(y) ]\subset \Re$. The support $S$ is partitioned into a set of contiguous
intervals (bins) $\{ I_1, \dots, I_h, \dots, I_H \}$, where
$I_h=[a_h;b_h)$ where $min(y)= a_1$ and $Max(y)=b_H$. Each $I_h$ is associated with a weight
$\pi_h$ that represents an empirical (or theoretical) relative frequency.

Let us $E$ be considered as a set of $n$ empirical distributions $F_i(y)$ ($i=1,
\dots, n$). In the case of a histogram description, it is possible to
assume that $S(i)=[min({y}_i);Max({y}_i)]$. Considering a set of
intervals $ I_{hi} = \left[{min({y}_{hi})
,Max({y}_{hi}) } \right) $ such that:
$$
i. \hspace{10pt}  I_{li}  \cap I_{mi}  = \emptyset ;\hspace{5pt}l \ne m \hspace{5pt}; \hspace{1cm}
ii. \hspace{10pt} \bigcup\limits_{s = 1,...,H_i} {I_{si} }  = [min({y}_i);Max({y}_i)] \\
$$
the support can also be written as  $S(i)  = \left\{
{I_{1i},...,I_{ui} ,...,I_{H_ii} } \right\}$. \\
In this paper, we
denote with $f_{i}(y)$ the (empirical) density function
associated with the description $y_i$ and with $F_i(y)$ its
distribution function. It is possible to define the description of the
$i-th$ histogram for the variable $Y$ as:

\begin{equation}
\begin{array}{l}
y_i=\{\left( {I_{1i}} ,\pi _{1i} \right),...,\left(
{I_{ui}} ,\pi _{ui} \right),...,\left( I_{H_i i} ,\pi _{H_i i}
\right)\} \\
such\:that\;
\forall I_{ui} \in S(i)\; \pi _{ui}  =
\int\limits_{I_{ui}} {f_{i} (y)dy}  \ge 0 \;and \;\int\limits_{S(i)} {f _{i} (y)dy}  = 1.
\end{array}
\end{equation}
In the following, we use $y_i$ to denote the description of the $i-th$ histogram in the univariate case.
If we obverse $p$ variables, we denote with $y_{ij}$ (where $i=1,\ldots,n$ and $j=1,\ldots,p$) $i-th$ histogram for the variable $j$.
Thus, considering to the classic data analysis approach, the individual$\times$variable input data table contains in each cell a histogram as represented in Table \ref{TabInput}.
\begin{table}[htbp]
  \centering
  \begin{tabular}{l|ccccc}
    \hline
    Objs. & Var 1 & \dots & Var j & \dots & Var p \\
    \hline
    1       & $y_{11}$     & \dots & $y_{1j}$ & \dots & $y_{1p}$ \\
    \dots   & \dots & \dots & \dots & \dots & \dots \\
    i       & $y_{i1}$= \includegraphics[width=2.0cm, height=1.0cm]{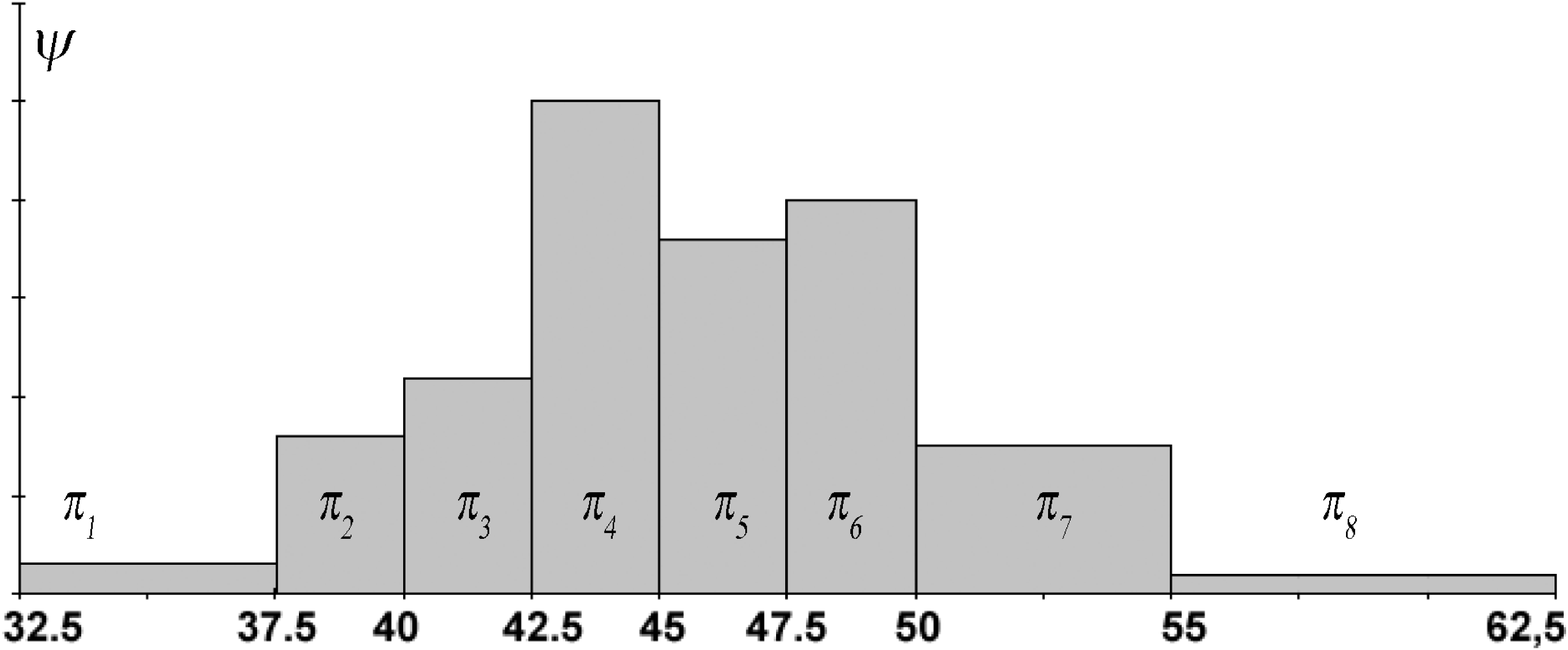}     & \dots & $y_{ij}=$\includegraphics[width=2.0cm, height=2.0cm]{fighist.eps}     & \dots & $y_{ip}=$\includegraphics[width=1.0cm, height=1.5cm]{fighist.eps}  \\
    \dots   & \dots & \dots & \dots & \dots & \dots \\ \\
    n       & $y_{n1}$     & \dots & $y_{nj}$ & \dots & $y_{np}$ \\
    \hline
  \end{tabular}
  \caption{Individual$\times$variable input data table of histogram data}\label{TabInput}
\end{table}

\subsection{The Wasserstein-Kantorovich distance between two histograms}
\label{sec:WKH}
The comparison of histogram data is a particular case of comparison of distribution functions. Several distances and dissimilarities have been presented in the literature. In \cite{VerIRP08did} we introduce several distances that can be used for comparing histogram data. Another good review on distances between distributions can be found in \cite{GiSu02}. The most part of these distances are based on the comparison of densities or frequency associated with a random variable, but the family of distances based on Wasserstein metric permits to obtain interesting interpretative results about the characteristics or the moments of the distributions (see \cite{VerIRP08did} for details).
Wasserstein-Kantorovich metric \cite{GiSu02}\cite{Villani03}, and, in particular the derived
$L^2$-Wasserstein distance is a natural extension of the Euclidean metric to compare distributions.
If $F_i(y)$ and
$F_{i'}(y)$ are the (empirical) distribution functions associated with  $y_i$ and $y_{i'}$,
respectively, and $F^{-1}_i(t)$ and
$F^{-1}_{i'}(t)$ ($t\in [0,1]$) their corresponding quantile functions, the $L^2$-Wasserstein metric is defined as follows:
\begin{equation}\label{HOMp}
    d_{W}(y_i,y_{i'}): =  \sqrt{\int\limits_0^1 {\left( {F_i^{ - 1} (t) - F_{i'}^{ - 1}
(t)} \right)^2dt}}
\end{equation}
This distance is also known as Mallows' \cite{Mall72} distance.
If the corresponding  quantile functions are centered by the respective sample means:
${\bar{F}_i^{-1}(t)=F_i^{-1}(t)-\bar{y}_{i}}$ and ${\bar{F}_{i'}^{-1}(t)=F_{i'}^{-1}(t)-\bar{y}_{i'}}$ ($\forall t\in[0,1]$), the corresponding histogram description are denoted with $y^c_i$ and $y^c_{i'}$ and \citet{CuALB97} proved that
\begin{equation} \label{eq:IrpRoma2} d^{2}_{W}(y_i,y_{i'})=(\bar{y}_{i}-\bar{y}_{i'})^{2}+d^{2}_{W}(y_i^c,y_{i'}^c)
\end{equation}
or, in other words, the (squared) Wasserstein distance between two
distributions $f_i(y)$ and $f_j(y)$ (or random variables), is equal to the
sum of the squared Euclidean distance between their means (the first
moments) and the squared Wasserstein distance between the two
centered random variables. The latter can be considered as a distance
measure of their dispersions, i.e., represent the difference between two distributions except for their location.

\citet{IrpinoR07} and \citet{VeIR08} showed that $L_2$ Wasserstein distance between two generic distributions
$F_i(y)$ and $F_{i'}(y)$ can be decomposed into the following components
(location, size and shape):

\begin{equation} \label{eq:IrpRoma} d^{2}_{W}(y_i, y_{i'})=\underbrace{(\bar{y}_{i}-\bar{y}_{i'})^{2}}_{Location}+\underbrace{\underbrace{(s_{i}-s_{i'})^{2}}_{Size}+\underbrace{2s_{i}s_{i'}\left[1-r_{QQ}(F^{-1}_i,F^{-1}_{i'})\right]}_{Shape}}_{d^{2}_{W}(y_i^c,y_{i'}^c)}
\end{equation}

where: $\bar{y}_i$, $s_i$ and $\bar{y}_{i'}$, $s_{i'}$ are the sample means and standard deviations respectively of $f_i(y)$ and $f_{i'}(y)$.
\begin{equation} \label{eq:rhoQQ}
r_{QQ}(F^{-1}_i,F^{-1}_{i'})=\frac{\int\limits_0^1 (F_i^{ - 1} (t) -
\bar{y}_{i})(F_{i'}^{ - 1} (t)-\bar{y}_{i'})dt} {s_{i}s_{i'}}=\\
\frac{\int\limits_0^1 F_i^{ - 1} (t) F_{i'}^{ - 1} (t)dt
-\bar{y}_{i}\bar{y}_{i'}}{s_{i}s_{i'}}
\end{equation}
is the sample correlation of the quantiles of the two empirical distributions as represented in a classical QQ plot.\\
A computational problem is related to the calculation of $r_{QQ}$, because it needs the computation of the quantile functions. \citet{IrVer06} showed that the the Wasserstein distance between histogram data depends only from the number of bins used for the histogram descriptions, avoiding the computational drawbacks related to the identification of the quantile functions of continuous distributions.

Wasserstein distance can be used for defining an inertia
measure among histograms like for the Euclidean metrics \cite{IrVer06}.
The total inertia with respect to the barycenter $g$ (which is a histogram whose quantiles are the means of the respective quantile of all the distributions) of a set of
$n$ histogram data is given by the following quantity:

\begin{equation}\label{INEDEC}
T=\sum\limits_{i=1}^{n}{d^2_{W}(y_i,g})
=\sum_{i=1}^{n}\int\limits_0^1
{\left( {F^{-1}_{i}(t)  -F^{-1}_{g}(t) } \right)^2dt}=\underbrace{\sum\limits_{k=1}^{K}{\sum\limits_{i \in C_k}{d^2_{W}(y_i,g_k)}}}_{W}+\underbrace{\sum\limits_{k=1}^{K}{|C_k|d^2_W(g_k,g)}}_{B}
.
\end{equation}

i.e.,  $T$ can be decomposed into within (W) and between (B) clusters inertia, according to the Huygen's theorem of decomposition,

where $g_k$ is the barycenter of the $k$-th cluster and $|C_k|$
is the number of objects in $E$ belonging to the cluster $C_k$. We will use this property for defining tools for the interpretation of clustering results.

\subsubsection{Multivariate and adaptive Wasserstein distance}
Given a set $E$ of $n$ objects described by $p$ variables as in Table \ref{TabInput},
each $y_{ij}$ is associated with a (empirical) density function $f_i(y_j)$, a distribution function $F_i(y_j)$ and a quantile function $F^{-1}_{i}(t_j)$. We denote with $\bar{y}_{ij}$ and $s_{ij}$ the sample mean and the sample standard deviation of $f_i(y_j)$.
The individual description of the $i-th$ object is then the following vector $\mathbf{y}_i$:
\begin{equation}\label{object}
    \mathbf{y}_i=[y_{i1},\ldots,y_{ip}].
\end{equation}
\citet{Clark84} present specific formulations for the multivariate Wasserstein distance when the distribution are Gaussians which are equipped with their covariance matrix. For all the other cases \cite{CuALB97}, it is not possible to compute analytically a Wasserstein distance between two multivariate distributions.
Considering that we do not know the joint histograms for each individual
 but only the marginal histograms, we consider the multivariate squared Wasserstein distance as follows:
\begin{equation}\label{dimult}
    d^2_W\left(\mathbf{y}_i,\mathbf{y}_{i'}\right)=\sum\limits_{j=1}^pd^2_W\left(y_{ij}, y_{i'j}\right).
\end{equation}
Following the same observations derived for Eqn. (\ref{eq:IrpRoma2}), we may rewrite the multivariate squared Wasserstein distance as follows:
\begin{equation}\label{multisep}
    d^2_W\left(\mathbf{y}_i,\mathbf{y}_{i'}\right)=\sum\limits_{j=1}^p\left(\bar{y}_{ij}-\bar{y}_{i'j}\right)^2
    +\sum\limits_{j=1}^pd^2_W\left(y^c_{ij}, y^c_{i'j}\right).
\end{equation}

In order to give a different weights to the variables we introduce adaptive distances \cite{DiGov77}, these are distances equipped with a system of weights. Let us consider a vector of weights $\Lambda=[\lambda^1,\ldots,\lambda^p]$ such that $\lambda^j>0$. According to \cite{DiGov77} and \cite{DECAYVES09}, a general formulation for an \emph{Adaptive Single Variable (squared) Wasserstein distance} is as follows:
\begin{equation}\label{multiada}
    d^2_W\left(\mathbf{y}_i,\mathbf{y}_{i'}|\Lambda\right)=\sum\limits_{j=1}^p\lambda^jd^2_W\left(y_{ij}, y_{i'j}\right).
\end{equation}
In this formulation, the weights induce a linear transformation of the original space. Several approaches of these types have been proposed (see for examples \cite{DECSOUZ07}, \cite{DECLEC09}, \cite{DeCDeS10}), where the weights are associated to the whole set of data or where the weights are chosen locally for each cluster in which is partitioned the set of data. If the distance is decomposable in several components, it is possible to introduce a suitable system of weights for such components. While the choice of weighting each variable is easily extendible to the histogram data, the choice of weighting components needs to be proven for such kind of data. In the present paper, starting from the decomposition of the squared Wasserstein distance as shown in Eqs. \ref{eq:IrpRoma2} and \ref{eq:IrpRoma}, we propose two schemas of weighting systems for the definition of two adaptive squared Wasserstein distances based on the two components: the first denoted as \emph{Globally Component-wise Adaptive Wassertein Distance} (GC-AWD), while the second is denoted as \emph{Cluster Dependent Component-wise Adaptive Wassertein Distance} (CDC-AWD).

\section{The Dynamic Clustering Algorithm} \label{sec:DCA}
The Dynamic Clustering Algorithm (DCA) \cite{Did71}\cite{DiSIM76} is here proposed as a method to partition a set of
data described by distributions. We recall that DCA is based on the
definition of a criterion of the best fitting between the partition
of a set of individuals and the representation of the clusters of
the partition. The algorithm simultaneously looks for the best
partition into $K$ clusters and their best representation. Thus, the
DCA needs the definition of a proximity function to assign the
individuals to the clusters and the definition of a way to represent the clusters.

The choice of the representative elements of the clusters (\emph{prototype}) is done according to the dissimilarity function used in the
algorithm to allocate the elements to the clusters, such that a
criterion of internal homogeneity is minimized. The consistence
between the representation and the allocation function guarantees
the convergence of the algorithm to a stationary value of the
criterion.

According to the nature of the data, we propose to base the adequacy criterion for the dynamic clustering algorithm on the Wasserstein distance between histogram data. We introduce two schemas of Dynamic Clustering algorithms based on two proposals of adaptive (squared) Wasserstein distances. In these algorithms, we propose two sets of weights computed on the whole set $E$ for defining the adequacy criterion on a \emph{Globally Component-wise Adaptive Wassertein Distance} (GC-AWD) and on a \emph{Cluster Dependent Component-wise Adaptive Wassertein Distance} (CDC-AWD).

\subsection{Adequacy criterion based on standard and adaptive distances}
\label{sec:ADCA}
Let us consider a set $E$ of $n$ objects described by $p$ histogram variables. The individual description of the $i-th$ object is $\mathbf{y}_i=[y_{i1},\ldots,y_{ip}]$, where $y_{ij}$ is the histogram description of the $i-th$ individual for the $j-th$ variable.
We assume that the prototype of the cluster $C_k\;(k=1,\ldots,K)$ is also represented by a vector $\mathbf{g}_k=(g_{k1},\ldots,g_{kp})$, where $g_{kj}$ is a histogram.
 As in the standard adaptive dynamic cluster algorithm, the proposed methods look for the partition $P=(C_1,\ldots,C_k)$ of $E$ in $K$ classes, its corresponding set of $K$  prototypes $\mathbf{G}=(\mathbf{g}_1,\ldots,\mathbf{g}_K)$ and a set of $K$ different adaptive distances $d=(d_1,\ldots,d_K)$ depending on a set $\mathbf{\Lambda}$ of positive weights associated with the clusters, such that the following adequacy criterion of the best fitting between the clusters and their representation is locally minimized:

\begin{equation}\label{gencrit}
    \Delta(\mathbf{G},\mathbf{\Lambda},P)=\sum\limits_{k=1}^K\sum\limits_{i\in C_k}d(\mathbf{y}_i,\mathbf{g}_k|\mathbf{\Lambda}).
\end{equation}

As in the standard adaptive dynamic cluster algorithm, we perform a representation step where the prototypes are updated according to the cluster structure of the data. This is followed by a weighting step, allowing the definition of the weight to be associated with each variable (or component) in the definition of the adaptive dissimilarity. Subsequently, an allocation step assigns the individuals to classes according to their proximity to the class prototypes. The three steps are repeated until convergence of the algorithm is achieved, i.e., until the adequacy criterion reaches a stationary value.\\

The adequacy criteria to be minimized are based on the following distances:

\begin{description}
\item[\textbf{STANDARD}] - \emph{Standard Wasserstein distance.}
 The standard (squared) Wasserstein distance
between the histogram $y_i$ and of the prototype $g_k$ is defined as:

\begin{equation}\label{STD_W}
        d(\mathbf{y}_i,\mathbf{g}_k)=\sum^{p}_{j=1} d^{2}_{W}(y_{ij},g_{kj}).
\end{equation}

In this case, the general criterion becomes:

\begin{equation}\label{critst}
    \Delta(\mathbf{G},P)=\sum\limits_{k=1}^K\sum\limits_{i\in C_k}d^2_W(\mathbf{y}_i,\mathbf{g}_k)
\end{equation}

as no system of weights is defined.

\item[\textbf{GC-AWD}] - \emph{Globally Component-wise Adaptive Wassertein Distance.} The distances depends from two vector $\mathbf{\Lambda}_{\bar y}$ and $\mathbf{\Lambda}_{Disp}$ of coefficients that assign weights for each component of each variable:%
   \begin{eqnarray}
      \mathbf{\Lambda}_{{\bar y}}&=&(\lambda^{1}_{\bar{y}},\ldots,\lambda^{p}_{\bar{y}})\\
    \mathbf{\Lambda}_{Disp}&=&(\lambda^{1}_{Disp},\ldots,\lambda^{p}_{Disp}).
   \end{eqnarray}
We define the two-component global adaptive Wasserstein distance
between the description $\mathbf{y}_i$ and prototype $\mathbf{g}_k$ as:
\begin{equation}\label{AD_2}
    d(\mathbf{y}_{i},\mathbf{g}_{k}|\mathbf{\Lambda})=\sum\limits^{p}_{j=1}\lambda^{j}_{{\bar y}}({\bar y}_{ij}-{\bar y}_{g_{kj}})^{2}+ \sum\limits^{p}_{j=1}\lambda^{j}_{Disp}d^2_W(y^c_{ij},g^c_{kj})
\end{equation}
where $y^c_{ij}$ and $g^c_{kj}$ are the centered description of $y_{ij}$ and of $g_{kj}$, as presented in Eqn. (\ref{eq:IrpRoma2}). In this case, being
\begin{equation}\label{LLK2}
 \mathbf{\Lambda}=\left[ {\begin{array}{*{20}c}
   \mathbf{\Lambda}_{{\bar y}}=(\lambda^{1}_{{\bar y}},\ldots,\lambda^{p}_{{\bar y}})  \\
   \mathbf{\Lambda}_{Disp}=(\lambda^{1}_{Disp},\ldots,\lambda^{p}_{Disp})
\end{array}} \right]
\end{equation}
the general criterion is:
\begin{equation}\label{critA2}
   \Delta(\mathbf{G},\mathbf{\Lambda},P)=\sum\limits_{k=1}^K\sum\limits_{i\in C_k}\sum\limits^{p}_{j=1}{\lambda^{j}_{{\bar y}}({\bar y}_{ij}-{\bar y}_{g_{kj}})^{2}+\sum\limits_{k=1}^K\sum\limits_{i\in C_k}\sum\limits^{p}_{j=1}\lambda^{j}_{Disp}d^2_W(y^c_{ij},g^c_{kj})}.
    \end{equation}

\item[\textbf{CDC-AWD}] - \emph{Cluster Dependent Component-wise Adaptive Wassertein Distance.} The distances depends from $K$ couples of vectors $\mathbf{\Lambda}_{k,{\bar y}}$ and $\mathbf{\Lambda}_{k,Disp}$ of coefficients that assign a weight for each component of the variable in each cluster.
 For each cluster we define:
    \[
\mathbf{\Lambda}_{k,{\bar y}}=(\lambda^{1}_{k,{\bar y}},\ldots,\lambda^{p}_{k,{\bar y}})\;\;\;
\mathbf{\Lambda}_{k,Disp}=(\lambda^{1}_{k,Disp},\ldots,\lambda^{p}_{k,Disp}).
\]
We define the two-component cluster-dependent adaptive Wasserstein distance
between the description $\mathbf{y}_i$ and prototype $\mathbf{g}_k$ as:
\begin{equation}\label{AD_4}
    d(\mathbf{y}_{i},\mathbf{g}_{k}|\mathbf{\Lambda})=\sum\limits^{p}_{j=1}
    \lambda^{j}_{k,{\bar y}}({\bar y}_{ij}-{\bar y}_{g_{kj}})^{2}+\sum\limits^{p}_{j=1}\lambda^{j}_{k,Disp}d^2_W(y^c_{ij},g^c_{kj})
\end{equation}
where $y^c_{ij}$ and $g^c_{kj}$ are the centered description of $y_{ij}$ and of $g_{kj}$, as presented in Eqn. (\ref{eq:IrpRoma2}). In this case, being
\begin{equation}\label{LL3K2}
 \mathbf{\Lambda}=\left[ {\begin{array}{*{20}c}
   \mathbf{\Lambda}_{k,{\bar y}}=\left(\lambda^{1}_{k,{\bar y}},\ldots,\lambda^{p}_{k,{\bar y}}\right)  \\
   \mathbf{\Lambda}_{k,Disp}=\left(\lambda^{1}_{k,Disp},\ldots,\lambda^{p}_{k,Disp}\right)
\end{array}} \right] \forall k\in\{1,\ldots,K\}
\end{equation}
the general criterion becomes:
\begin{equation}\label{critA2}\begin{array}{l}
   \Delta(\mathbf{G},\mathbf{\Lambda},P)=\sum\limits_{k=1}^K\sum\limits_{i\in C_k}\sum\limits^{p}_{j=1}{\lambda^{j}_{k,{\bar y}}({\bar y}_{ij}-{\bar y}_{g_{kj}})^{2}+} \\
    {+\sum\limits_{k=1}^K\sum\limits_{i\in C_k}\sum\limits^{p}_{j=1}\lambda^{j}_{k,Disp}d^2_W(y^c_{ij},g^c_{kj})}.
    \end{array}
    \end{equation}
\end{description}
From an initial solution for $(\mathbf{G}^0,\mathbf{\Lambda}^0,P^0)$, the dynamic clustering algorithm based on adaptive distances alternates three steps until the criterion $\Delta$ reach a stationary point.
\subsubsection{Step 1: definition of the prototypes}
In the first step of the algorithm, the partition $P$ of $E$ in $K$ clusters and the corresponding weights $\mathbf{\Lambda}$ are fixed.
\begin{proposition}\label{prop01}
Chosen one of the distance functions (Eqs. \ref{STD_W},  \ref{AD_2} or \ref{AD_4}), the vector of prototypes $\mathbf{G}=(\mathbf{g}_1,\ldots,\mathbf{g}_k)$, where $\mathbf{g}_k=(g_{k1},\ldots,g_{kp})$, which minimizes the criterion $\Delta$, is calculated by means of the quantile functions associated with each $g_{kj}$. $g_{kj}$ is represented by a histogram with associated the following quantile function:

\begin{equation}\label{barycenter}
F^{-1}_{g_{k}}(t_j)=\frac{1}{n_{k}}\sum\limits_{i\in C_k}F^{-1}_{i}(t_j)\;\;\;\forall t_j\in[0, 1]
\end{equation}
where $n_k$ is the cardinality of cluster $C_k$. It is the means of the quantile functions of the histogram representations of the elements belonging to the cluster $C_k$ and it can also be written as:
\begin{equation}\label{barycenterC}
F^{-1c}_{g_{k}}(t_j)+{\bar y}_{g_{kj}}=\frac{1}{n_{k}}\sum\limits_{i\in C_k}F^{-1c}_{i}(t_j)+\frac{1}{n_{k}}\sum\limits_{i\in C_k}{\bar y}_{ij}\;\;\;\forall t_j\in[0, 1]
\end{equation}
where $F^{-1c}_{g_{k}}(t_j)$ (resp. $F^{-1c}_{i}(t_j)$) is associated with the centered description $g^c_{kj}$  (resp. $y^c_{ij}$) of $g_{kj}$(resp. $y_{ij}$) and ${\bar y}_{g_{kj}}$  (resp. ${\bar y}_{ij}$) is the mean of the description of $g_{kj}$ (resp. $y_{ij}$).
\end{proposition}
\begin{proof}
Beeing $\Delta$ an additive criterion (for the variables, the components and the clusters), according to \cite{CuALB97}, we can write:
$$\begin{array}{l}
min\left(\phi(g_{kj})\right)=
\underbrace{min\left\{\sum\limits_{i \in C_k}\int\limits_0^1\left(F^{-1}_{i}(t_j)-F^{-1}_{g_k}(t_j)\right)^2dt_j\right\}}_{I}=\\
=
min\left\{\sum\limits_{i \in C_k}\left({\bar y}_{ij}-{\bar y}_{g_{kj}}\right)^2+\sum\limits_{i \in C_k}\int\limits_0^1\left({F}^{-1c}_{i}(t_j)-{F}^{-1c}_{g_k}(t_j)
\right)^2dt_j\right\}=\\
=\underbrace{min\left\{\sum\limits_{i \in C_k}\left({\bar y}_{ij}-{\bar y}_{g_{kj}}\right)^2\right\}}_{II}
+\underbrace{min\left\{\sum\limits_{i \in C_k}\int\limits_0^1\left({F}^{-1c}_{i}(t_j)-{F}^{-1c}_{g_k}(t_j)
\right)^2dt_j\right\}}_{III}
\end{array}$$
where ${F}^{-1c}_{i}(t_j)=F^{-1}_{i}(t_j)-{\bar y}_{ij}$ and ${F}^{-1c}_{g_k}(t_j)=F^{-1}_{g_k}(t_j)-{\bar y}_{g_{kj}}$.
Problem $II$ is minimized when:
$${\bar y}_{g_{kj}}=\frac{1}{n_k}\sum\limits_{i \in C_k}{\bar y}_{ij}\;\;\;;$$
while problem $III$ is minimized when, for each $t_j\in[0,1]$:
$${F}^{-1c}_{g_k}(t_j)=\frac{1}{n_k}\sum\limits_{i \in C_k}{F}^{-1c}_{i}(t_j).$$
Then, problem $I$ is minimized when the barycenter (\emph{prototype}) histogram $g_{kj}$ is a histogram whose quantile function is:
\begin{equation}\label{bary}
    F^{-1}_{g_k}(t_j)={F}^{-1c}_{g_k}(t_j)+{\bar y}_{g_{kj}}\;\;\;\forall\; t_j \in [0,1].
\end{equation}
\end{proof}

\subsubsection{Step 2: definition of the best distances}
In the second step, the partition $P$ of $E$ and the corresponding vector $\mathbf{G}$ of the prototypes are fixed. It is worth noting that the Dynamic Clustering Algorithm based on the standard (squared) Wasserstein distances does not require this step.
\begin{proposition}\label{prop02}
According to \citet{DiGov77}, the $\mathbf{\Lambda}$ system of weights, useful for defining the adaptive distances that minimize the criterion $\Delta$, is calculated accordingly to the different schemes of adaptive distance functions used for the definition of the adequacy criterion of the algorithm.
\begin{description}

\item[\textbf{GC-AWD}]If the distance function is given by Eqn. (\ref{AD_2}), the $\mathbf{\Lambda}$ system of weights that minimize the criterion $\Delta$ is

\begin{equation}\label{LL3K}
 \mathbf{\Lambda}=\left[ {\begin{array}{*{20}c}
   \mathbf{\Lambda}_{{\bar y}}=(\lambda^{1}_{{\bar y}},\ldots,\lambda^{p}_{{\bar y}})  \\
   \mathbf{\Lambda}_{Disp}=(\lambda^{1}_{Disp},\ldots,\lambda^{p}_{Disp})
\end{array}} \right]
\end{equation}
under the following restrictions:
\begin{enumerate}
    \item $\lambda^{j}_{\bar y}>0$ and $\lambda^{j}_{Disp}>0$
    \item $\prod^{p}_{j=1}\lambda^{j}_{\bar y}=1$ and $\prod^{p}_{j=1}\lambda^{j}_{Disp}=1$.
\end{enumerate}

The coefficients of $\mathbf{\Lambda}_{{\bar y}}$ and $\mathbf{\Lambda}_{Disp}$, which satisfy restrictions $(1)$ and $(2)$ and minimize the criterion $\Delta$, are calculated as:

\begin{eqnarray}
 \lambda^{j}_{{\bar y}}&=&\frac{\left[\prod^{p}_{h=1}\sum\limits_{k=1}^K \sum\limits_{i \in C_k}({\bar y}_{ih}-{\bar y}_{g_{kh}})^{2}\right]^{\frac{1}{p}}}{\sum\limits_{k=1}^K \sum\limits_{i \in C_k}({\bar y}_{ij}-{\bar y}_{g_{kj}})^{2}},\label{eq:lambda2_1}\\
 \lambda^{j}_{Disp}&=&\frac{
\left[\prod^{p}_{h=1}\left(\sum\limits_{k=1}^K{\sum\limits_{i \in C_k}d^{2}_{W}(y^c_{ih} ,g^c_{kh})}\right)\right]^{\frac{1}{p}}
}
{
\sum\limits_{k=1}^K \sum\limits_{i \in C_k}d^{2}_{W}(y^c_{ij} ,g^c_{kj})
}.\label{eq:lambda2_2}
\end{eqnarray}

Note that the closer are the objects to the prototypes of the clusters for the component related to the mean (resp. dispersion) the higher is the respective weight.

\item[\textbf{CDC-AWD}]If the distance function is given by Eqn. (\ref{AD_4}), the $\mathbf{\Lambda}$ system of weights that minimize the criterion $\Delta$ is:
    \begin{equation}\label{LL4}
\mathbf{ \Lambda}=\left[ {\begin{array}{*{20}cc}
\mathbf{ \Lambda}_{1,{\bar y}} &\mathbf{\Lambda}_{1,Disp} \\
\ldots&\ldots\\
  \mathbf{ \Lambda}_{k,{\bar y}} &\mathbf{\Lambda}_{k,Disp} \\
  \ldots&\ldots\\
   \mathbf{ \Lambda}_{K,{\bar y}} &\mathbf{\Lambda}_{K,Disp}
   \end{array}}\right]
\end{equation}

where, for each cluster, $k=1,\ldots,K$:
    \[
\mathbf{\Lambda}_{k,{\bar y}}=(\lambda^{1}_{k,{\bar y}},\ldots,\lambda^{p}_{k,{\bar y}})\;\;\;
\mathbf{\Lambda}_{k,Disp}=(\lambda^{1}_{k,Disp},\ldots,\lambda^{p}_{k,Disp})
\]
under the following restrictions for each cluster $k=1,\ldots,K$:
\begin{enumerate}
    \item $\lambda^{j}_{k,{\bar y}}>0$ and $\lambda^{j}_{k,Disp}>0$
    \item $\prod^{p}_{j=1}\lambda^{j}_{k,{\bar y}}=1$ and $\prod^{p}_{j=1}\lambda^{j}_{k,Disp}=1.$
\end{enumerate}
The coefficients $\lambda^{j}_{k,{\bar y}}$ and $\lambda^{j}_{k,Disp}$, which satisfy restrictions $(1)$ and $(2)$ and minimize the criterion $\Delta$, are:
\begin{eqnarray}
  \lambda^{j}_{k,{\bar y}}& = &\frac{[\prod^{p}_{h=1}\sum_{i\in C_k}({\bar y}_{ih}-{\bar y}_{g_{kh}})^{2}]^{\frac{1}{p}}}{\sum_{i\in C_k}({\bar y}_{ij}-{\bar y}_{g_{kj}})^{2}}, \;\;\mathrm{and} \label{lambda4_1}\\
  \lambda^{j}_{k,Disp}&=&\frac{\left[\prod^{p}_{h=1}\sum_{i\in C_k}d_W^2(y^c_{ij},g^c_{kj})\right]^{\frac{1}{p}}}
  {\sum_{i\in C_k}d_W^2(y^c_{ij},g^c_{kj})}.\label{lambda4_2}
\end{eqnarray}
Note that the closer to the prototypes of a given cluster $C_k$ are the objects for the component related to the mean (resp. dispersion) the higher is the respective weight for the cluster $C_k$.
\end{description}
\end{proposition}
\begin{proof}
The proof of the proposition can be performed according to \cite{DiGov77} using the Lagrange multipliers method for the minimization of the criterion. For the sake of brevity, and without loss of generality, see \cite{DECAYVES09} and \cite{DeCDeS10}.
\end{proof}

\subsubsection{Step 3: definition of the best partition}
In this step, prototype $\mathbf{G}$ and the $\mathbf{\Lambda}$  are fixed.
\begin{proposition}\label{prop03}
The partition $P=(C_1,\ldots,C_K)$, which minimizes the criterion $\Delta$, consists of clusters $C_k\;(k=1,\ldots,K)$ identified according to the following allocation rules:
\begin{description}
\item[\textbf{STANDARD}]
\begin{equation}\label{ass1}
C_k=\left\{i\in E| d(\mathbf{y}_i, \mathbf{g}_k) <d(\mathbf{y}_i, \mathbf{g}_m) \right\}\;\;\forall m \neq k \;(m=1,\ldots,K).
\end{equation}
\item[\textbf{GC-AWD} and \textbf{CDC-AWD}]
\begin{equation}\label{ass2}
C_k=\left\{i\in E| d(\mathbf{y}_i, \mathbf{g}_k|\mathbf{\Lambda}) <d(\mathbf{y}_i, \mathbf{g}_m|\mathbf{\Lambda}) \right\}\;\;\forall m \neq k \;(m=1,\ldots,K).
\end{equation}
\end{description}

In general, when $d(\mathbf{y}_i, \mathbf{g}_k) =d(\mathbf{y}_i, \mathbf{g}_m)$ or $d(\mathbf{y}_i, \mathbf{g}_k|\mathbf{\Lambda}) =d(\mathbf{y}_i, \mathbf{g}_m|\mathbf{\Lambda})$, then $i\in C_k$ if $k<m$.
\end{proposition}
\begin{proof}
The proof of Proposition \ref{prop03} is straightforward.
\end{proof}

\subsection{Properties of the algorithm}
According to \cite{DiSIM76}, the property of convergence of the algorithm can be studied from two series: $z_t=(\mathbf{G}^t,\mathbf{\Lambda}^t,P^t)$ and $ss_t=\Delta(z_t)$, $t=0,1,\ldots$\\
From an initial term $z_0=(\mathbf{G}^0,\mathbf{\Lambda}^0,P^0)$, the algorithm computes the different terms of the series $z_t$ until convergence, at which point the criterion $\Delta$ achieves a stationary value.
Here, we show the convergence of algorithms in terms of configuration of partition, prototypes and criterion $\Delta(\mathbf{G},\mathbf{\Lambda},P)$.
\begin{proposition}\label{prop5}
 The series $ss_t=\Delta(z_t)$ decreases at each iteration and converges.
\end{proposition}
\begin{proof}First, we show that inequalities $(I)$, $(II)$ and $(III)$
\begin{equation}\label{pr1}
\underbrace {\Delta\left( {\mathbf{G}^t ,\mathbf{\Lambda} ^t ,P^t } \right)}_{ss_t }\overbrace {\: \ge \:}^{(I)}\Delta\left( {\mathbf{G}^{t + 1} ,\mathbf{\Lambda} ^t ,P^t } \right)\overbrace {\: \ge \:}^{(II)}\Delta\left( {\mathbf{G}^{t + 1} ,\mathbf{\Lambda}^{t + 1} ,P^t } \right)\overbrace {\: \ge \:}^{(III)}\underbrace {\Delta\left( {\mathbf{G}^{t + 1} ,\mathbf{\Lambda} ^{t + 1} ,P^{t + 1} } \right)}_{ss_{t + 1} }
\end{equation}
hold.\\
Inequality $(I)$ holds because
$$
\Delta\left( {\mathbf{G}^t ,\mathbf{\Lambda} ^t ,P^t } \right)=\sum\limits_{k=1}^K\sum\limits_{i\in C_k^t }d(\mathbf{y}_i,\mathbf{g}^t_k|\mathbf{\Lambda}^t)
$$
$$
\Delta\left( {\mathbf{G}^{t+1} ,\mathbf{\Lambda} ^t ,P^t } \right)=\sum\limits_{k=1}^K\sum\limits_{i\in C_k^t }d(\mathbf{y}_i,\mathbf{g}^{t+1}_k|\mathbf{\Lambda}^t)
$$
and according to Proposition \ref{prop01},
$$
\mathbf{g}_k^{(t + 1)}  = \underbrace {argmin}_{\mathbf{g_k} }\sum\limits_{j = 1}^p {\sum\limits_{i \in C_k^{(t)} } {d\left( {y_{ij} ,g_{kj}^{(t)} |\mathbf{\Lambda}^{(t)} } \right) } } \, (k=1, \dots, K).
$$
Inequality (II) also holds because
$$\Delta\left( {\mathbf{G}^{t+1} ,\mathbf{\Lambda} ^{t+1} ,P^t } \right)=\sum\limits_{k=1}^K\sum\limits_{i\in C_k^t }d(\mathbf{y}_i,\mathbf{g}^{t+1}_k|\mathbf{\Lambda}^{t+1})
$$
and according to Proposition \ref{prop02},
$$
    \lambda^{(t+1)}_k=\underbrace {argmin}_{\lambda_{k}}\sum\limits_{i\in C_k^{(t)}}{d\left( {y_{ij} ,g_{kj}^{(t+1)} |\lambda_k } \right)}.
$$
Finally, inequality $(III)$ holds because
$$\Delta\left( {\mathbf{G}^{t+1} ,\mathbf{\Lambda} ^{t+1} ,P^{t+1} } \right)=\sum\limits_{k=1}^K\sum\limits_{i\in C_k^t }d(\mathbf{y}_i,\mathbf{g}^{t+1}_k|\mathbf{\Lambda}^{t+1})
$$
and according to Proposition \ref{prop03}
\begin{equation}\label{C_ottimo}
    C^{(t+1)}_k=\underbrace {argmin}_{C\in\mathbb{P}(E)}\sum\limits_{i\in C}{d\left( {y_{ij} ,g_{kj}^{(t+1)} |\Lambda^{(t+1)} } \right)} \; (k=1,\ldots,K).
\end{equation}
Finally, because the series $ss_t$ decreases and it is bounded ($\Delta(z_t)\geq 0$), it converges.
\end{proof}
\begin{proposition}\label{prop6}
 The series $z_t=(G^t,\Lambda^t,P^t)$ converges.
\end{proposition}
\begin{proof}
Let us assume that the stationary point of the series $ss_t$ is achieved at the iteration $t=T$. Then, we have that $z_T=z_{T+1}$ and then $\Delta(z_T)=\Delta(z_{T+1})$ (i.e., $ss_T=ss_{T+1}$ ).

From  $\Delta(z_T)=\Delta(z_{T+1})$, we have $\Delta(\mathbf{G}^T,\mathbf{\Lambda}^T,P^T)=\Delta(\mathbf{G}^{T+1},\mathbf{\Lambda}^{T+1},P^{T+1})$ and this equality, according to Proposition \ref{prop5}, can be rewritten as equalities $(I)$, $(II)$ and $(III)$ as follows:
\begin{equation}\label{pr1}
\begin{array}{l}
\Delta\left( {\mathbf{G}^T ,\mathbf{\Lambda} ^T ,P^T } \right)\overbrace {=}^{(I)}\Delta\left( {\mathbf{G}^{T + 1} ,\mathbf{\Lambda} ^T ,P^T } \right)\overbrace {=}^{(II)}\\
\overbrace {=}^{(II)}\Delta\left( {\mathbf{G}^{T + 1} ,\mathbf{\Lambda} ^{T + 1} ,P^T } \right)\overbrace {=}^{(III)}\Delta\left( {\mathbf{G}^{T + 1} ,\mathbf{\Lambda} ^{T + 1} ,P^{T + 1} } \right).
\end{array}
\end{equation}
From the first equality $(I)$, we know that $\mathbf{G}^T=\mathbf{G}^{T+1}$ because $\mathbf{G}$ is unique, minimizing $\Delta$ when the partition $P^T$ and $\mathbf{\Lambda}^T$ are fixed. From the second equality $(II)$, we know that $\mathbf{\Lambda}^T=\mathbf{\Lambda}^{T+1}$ because $\mathbf{\Lambda}$ is unique, minimizing $\Delta$ when the partition $P^T$ and $\mathbf{G}^{T+1}$ are fixed. Then, from the third equality $(III)$, we know that $P^T=P^{T+1}$ because $P$ is unique, minimizing $\Delta$ when $\mathbf{\Lambda}^{T+1}$ and $\mathbf{G}^{T+1}$ are fixed.

Finally, we conclude that $z_T=z_{T+1}$. This conclusion holds for all $t\geq T$ and $z_t=z_T$, $\forall\;t\geq T$ and then $z_t$ converges.
\end{proof}
\subsubsection{Complexity of the algorithm}
Let us denote the number of operations required for the computation of the Wasserstein distance with $D$, the number of operations for computing the mean quantile function of a set of distributions with $Q$ and the number of iterations of the algorithm as $I$. The representation step requires approximatively $O(np\times Q)$ operations. The weighting step requires approximatively $O(np\times D)$ operations; indeed it is based on the computation of the inertia within the clusters. The allocation step requires approximatively $O(nKp\times D)$ operations. Then, the complete computational cost is of order $O\left(nKp\times D\times I\right)$.
\subsection{Algorithm schema}
The procedure steps for the Dynamic Clustering Algorithm with adaptive distances for histogram data are described in algorithm \ref{algo2}.
\begin{algorithm}[htbp]
\caption{Dynamic clustering with adaptive distances}
\label{algo2}
\begin{algorithmic}
\REQUIRE{K number of clusters, \\$E$ a set $n\geq K$ individuals described by $p>0$ histogram variables }
\STATE
\STATE \textbf{Initialization}
\STATE Set $t=0$
\STATE Randomly choose a partition of $E$ into $K$ clusters $P^t=(C^t_1,\ldots,C^t_k)$ of $E$
\STATE Set $\mathbf{\Lambda}^t=\mathbf{1}$.
\STATE
\REPEAT
\STATE \textbf{Step 1: Definition of the best prototypes (Representation step)}
\\
Given $P^t$ and $\mathbf{\Lambda}^t$, compute the $\mathbf{G}^{t+1}$ set of $K$ prototypes according to the criterion $\Delta(\mathbf{G}^{t+1},\mathbf{\Lambda}^t,P^t)$ in Eqn. (\ref{barycenter}).

\STATE
\STATE \textbf{Step 2: Definition of the best distances (Weighting step)}
\\Given $P^t$ and $\mathbf{G}^{t+1}$, compute the $\mathbf{\Lambda}^{t+1}$ matrix according to Eqs. (\ref{eq:lambda2_1}) and (\ref{eq:lambda2_2}) for \textbf{GC-AWD},  and Eqs. (\ref{lambda4_1}) and (\ref{lambda4_2}) for \textbf{CDC-AWD}.
\STATE
\STATE \textbf{Step 3: Definition of the best partition (Allocation step)}
\\Given $\mathbf{G}^{t+1}$ and $\mathbf{\Lambda}^{t+1}$ allocate each element of $E$ to the closest prototype according to Eqs. (\ref{ass1}) or (\ref{ass2})\\
$test\leftarrow 0$\\
$P^{t+1}\leftarrow P^{t}$\\
\FOR {$i=1$ to $n$}
\STATE find the cluster $C^{t+1}_m$ to which $i$ belongs
\STATE find the winning cluster $C^{t+1}_k$ such that
\STATE $k=\mathop{argmin}\limits_{1\leq h \leq K}d(\mathbf{y}_i,\mathbf{g}^{t+1}_h|\mathbf{\Lambda})$
\IF {$k\neq m$}
\STATE $test\leftarrow 1$
\STATE $C^{t+1}_k\leftarrow C^{t+1}_k\cup \{i\}$
\STATE $C^{t+1}_m\leftarrow C^{t+1}_m\backslash \{i\}$
\ENDIF
\ENDFOR
\STATE
\STATE set $t=t+1$
\UNTIL {$test\neq 0$}
\end{algorithmic}
\end{algorithm}

\section{Tools for the interpretation of the partition}\label{sec:tools}
After a clustering task, it is important to evaluate the results of the procedure, in order to have information about the intra-class and the inter-classes heterogeneity, the contribution of each variable to the final partition, etc. In the framework of the standard Dynamic Cluster Algorithm, some indices, based on the ratio between the intra-cluster and the total inertia, can be used as tools for interpreting the clustering results  \citet{CelAL89}.
Such indeces have been extended \citet{deca06b} to the case of dynamic clustering with adaptive and non-adaptive Euclidean distances for interval-valued data. Here, we use the same class of indices for the evaluation of the partition of histogram data.

In the following section, we define an extension of:
\begin{description}
\item[Total inertia]. It is the total sum of squared (TSS) distances of all the elements of the set $E$ from the global prototype ($g_E$). TSS is the adequacy criterion;
\item[Within inertia]. It is the within sum of squared (WSS) distances of the elements of a cluster from the respective barycenter, for all the clusters of the partition. It corresponds to the criterion $\Delta\left( {\mathbf{G},\mathbf{\Lambda},P}\right)$ with $P={C_1, \dots, C_k, \dots, C_K}$; $\mathbf{G}={g_1,\dots, g_k,\dots, g_K}$ and
\begin{equation}\label{LL3K}
 \mathbf{\Lambda}=\left[ {\begin{array}{*{20}c}
   \mathbf{\Lambda}_{{\bar y}}=(\lambda^{1}_{{\bar y}},\ldots,\lambda^{p}_{{\bar y}})  \\
   \mathbf{\Lambda}_{Disp}=(\lambda^{1}_{Disp},\ldots,\lambda^{p}_{Disp})
\end{array}} \right]
\end{equation}

according to the \textbf{GC-AWD};

as well as

\begin{equation}\label{LL4}
\mathbf{ \Lambda}=\left[ {\begin{array}{*{20}cc}
\mathbf{ \Lambda}_{1,{\bar y}} &\mathbf{\Lambda}_{1,Disp} \\
\ldots&\ldots\\
  \mathbf{ \Lambda}_{k,{\bar y}} &\mathbf{\Lambda}_{k,Disp} \\
  \ldots&\ldots\\
   \mathbf{ \Lambda}_{K,{\bar y}} &\mathbf{\Lambda}_{K,Disp}
   \end{array}}\right]
\end{equation}

according to the \textbf{CDC-AWD};

\item[Between inertia]. It is the sum of squared (BSS) distances between the prototypes of the several clusters and the general prototypes.
  Using the adaptive Wasserstein distance, and according to the Huygen's theorem of decomposition of inertia, we prove that the total sum of squares can be decomposed into two additive terms: the within sum of squares and the between sum of squares (TSS=WSS+BSS).
\end{description}

\subsection{Total sum of squares (TSS)}
Let $E$ be a set of $n$ real points, denoted as $x_i$, where each of them is weighted by a positive real number denoted as $\pi_i$ and a positive number $h$:
\begin{equation}\label{th_1}
    \sum\limits_{i=1}^n\sum\limits_{j>i} {\pi_i \pi_j (x_i-x_j)^2}=h\sum\limits_{i=1}^n {\pi_i (x_i-{\bar x})^2}
\end{equation}
where
\begin{equation}\label{MU}
{\bar x}=\frac{\sum\limits_{i=1}^n {\pi_i x_i}}{\sum\limits_{i=1}^n \pi_i}
\end{equation}
and
$h=\sum\limits_{j=1}^n\pi_j.
$
We define the total sum of squares as:
\begin{equation}
TSS=\sum\limits_{i=1}^n{\pi_i \left(x_i-{\bar x}\right)^2}.
\end{equation}
If we have $p>1$ descriptors, the total sum of squares is:
\begin{equation}\label{TSSR}
TSS=\sum\limits_{j=1}^p \sum\limits_{i=1}^n{\pi_{ij} \left(x_{ij}-{\bar x}_j\right)^2}.
\end{equation}
Without loss of generality, given a set $E$ of $n$ individuals described by $p$ histogram descriptors, in our two adapting schemes of Wasserstein distance, for a set of data clustered into $K$ groups, we have two kinds of $TSS$ and global prototypes that are consistent with Eqn. (\ref{barycenter}) and Eqn. (\ref{TSSR}).
\begin{description}
  \item[\textbf{STANDARD}] According to Eq. \ref{barycenter}, the general prototype $\mathbf{g}_{E}=(g_{E1},\ldots,g_{Ep})$ is a vector of histogram descriptions $g_{Ej}$ whose quantile function is:
  \begin{equation}\label{genprot_A1}
F^{-1}_{g_E}(t_j)=n^{-1}\sum\limits_{i=1}^n{F^{-1}_{i}(t_j)}\;\;\forall t_j\in[0,1]
\end{equation}

i.e., it is the histogram whose quantiles are the average of the quantiles of the elements of $E$ for the $j-th$ variable and the total sum of squares is
\begin{equation}\label{TSQ_A1}
    TSS=\sum\limits_{j=1}^p\sum\limits_{i\in E}d^2_{W}(y_{ij},g_{Ej}).
\end{equation}
  \item[\textbf{GC-AWD}] The general prototype is computed according to Eqs. \ref{barycenterC} and \ref{MU}.
      Thus, the general prototype $\mathbf{g}_{E}=(g_{E1},\ldots,g_{Ep})$ is a vector of histogram descriptions $g_{Ej}$ whose quantile function is:
      \begin{equation}\label{genprot_A1}
F^{-1}_{g_E}(t_j)=\underbrace{n^{-1}\sum\limits_{i=1}^n{\bar{y}_{ij}}}_{{\bar y}_{g_{Ej}}}+\underbrace{n^{-1}\sum\limits_{i=1}^n{F^{-1c}_{i}(t_j)}}_{F^{-1}_{g^c_E}(t_j)}\;\;\forall t_j\in[0,1]
\end{equation}

the total sum of squares is
\begin{equation}\label{TSQ_A2}
    TSS_{GC-AWD}=\sum\limits_{j=1}^p\sum\limits_{i\in E}\left[\lambda^j_{\bar y}({\bar y}_{ij}-{\bar y}_{g_{Ej}})^2 +\lambda^j_{Disp}d^2_{W}(y^c_{ij},g^c_{Ej})\right].
\end{equation}

\item[\textbf{CDC-AWD}] The general prototype $\mathbf{g}_{E}=(g_{E1},\ldots,g_{Ep})$ is a vector of histogram descriptions $g_{Ej}$ whose quantile function is:
    \begin{equation}\label{a4PRO}
        F^{-1}_E(t_j)=\bar{F}^{-1}_E(t_j)+{\bar y}_{g_{Ej}}
    \end{equation}
    where
  \begin{equation}\label{genprot_A4}
  \begin{array}{l}
      {\bar y}_{g_{Ej}}=\sum\limits_{k=1}^K\frac{\lambda^j_{k,{\bar y} }}{\sum\limits_{h=1}^K n_h\lambda^j_{h,{\bar y}}}\sum\limits_{i\in C_k}{\bar y}_{ij}\;\;\; and\\
    \bar{F}^{-1}_{g_E}(t_j)=\sum\limits_{k=1}^K\frac{\lambda^j_{k,Disp}}{\sum\limits_{h=1}^K n_h\lambda^j_{h,Disp}}\sum\limits_{i\in C_k}\bar{F}^{-1}_i(t_j)\;\;\forall t_j \in [0,1];
    \end{array}
\end{equation}
where $n_h$ is the number of elements belonging to the cluster $h$. In this case, the total sum of squares is
\begin{equation}\label{TSQ_A4}
    {TSS}_{CDC-AWD}=\sum\limits_{j=1}^p\sum\limits_{k=1}^K\sum\limits_{i\in C_k}\left[\lambda^j_{k,{\bar y}}\left({\bar y}_{ij}-{\bar y}_{g_{Ej}}\right)^2+ \lambda^j_{k,Disp}d^2_{W}(y^c_{ij},g^c_{Ej})\right];
\end{equation}
\end{description}

\subsection{Within sum of squares WSS}
The within sum of squares corresponds to the minimized criterion in the algorithm. Therefore, we recall the same results obtained in section \ref{sec:ADCA} for the proposed two adaptive distances:
  \begin{equation}\label{W}
    WSS=\Delta(\mathbf{G},\mathbf{\Lambda},P).
    \end{equation}

\subsection{Between sum of squares BSS}
For the proposed two adaptive distances, the between sum of squares are:
\begin{description}
 \item[\textbf{STANDARD}]
  \begin{equation}\label{BSQST}
    BSS=\sum\limits_{j=1}^p\sum\limits_{k=1}^K n_k d^2_{W}(g_{kj},g_{Ej});
 \end{equation}
  \item[\textbf{GC-AWD}]
  \begin{equation}\label{BSQ2}
    BSS_{GC-AWD}=\sum\limits_{j=1}^p\sum\limits_{k=1}^K n_k\left[\lambda^j_{\bar y}({\bar y}_{g_{kj}}-{\bar y}_{g_{Ej}})^2 +\lambda^j_{Disp}d^2_{W}(g^c_{kj},g^c_{Ej})\right];
 \end{equation}

  \item[\textbf{CDC-AWD}]
   \begin{equation}\label{BSQ4}
    BSS_{CDC-AWD}=\sum\limits_{j=1}^p\sum\limits_{k=1}^K n_k \left[\lambda^j_{k,{\bar y}}\left({\bar y}_{g_{kj}}-{\bar y}_{g_{Ej}}\right)^2+ \lambda^j_{k,Disp}d^2_{W}(g^c_{kj},g^c_{Ej})\right].
 \end{equation}
\end{description}
Based on the quantities $TSS$, $WSS$ and $BSS$, we can evaluate the clustering procedure as in \cite{deca06b}, where a quality of partition index is considered and is defined as:
\begin{equation}\label{PQI}
    QPI=1-\frac{WSS}{TSS}.
\end{equation}
Holding the  decomposition of inertia for the Wasserstein distance, the $QPI$ index can be written as follows:
\begin{equation}\label{QPI}
    QPI=\frac{BSS}{TSS}.
\end{equation}
Equation \ref{PQI} is equal to eq. \ref{QPI} if and only if the total inertia can be decomposed in two components: the inter-cluster inertia ($BSS$) and the intra-cluster inertia ($WSS$). We showed in Eq. \ref{INEDEC} that this is true for Wasserstein distance. Considering the additivity of $TSS$, $WSS$ and $BSS$, it is possible to detail the $QPI$ for each variable, for each cluster and for each component.

\section{Experimental results}\label{sec:ExperimentalResults}
Clustering (considered an unsupervised learning task) is an explorative method applied to a dataset, where, in general, no information about the class structure is available to the researcher. Agreeing with \citet{Meila05}, there are many competing criteria for comparing clusterings, with no clear best choice.
In an experimental evaluation, the quality of clustering algorithms is often based on their performance according to a specific quality index. Experiments use either a limited number of
real-world instances or synthetic data. While real-world data is crucial for testing the proposed algorithms,
until now there is a lack of public repositories furnishing data described by histograms. This is indeed true for histogram data; histograms are composed by synthesized raw data, while a large repository of standard data exists. Therefore, a test bed of synthetic, pre-classified data must be assembled by a generator.\\
In order to asses the quality of the proposed algorithms, we used synthetic data and histogram representations of real data. For the analysis of the synthetic data, we set up two Monte Carlo experiments that allowed the generation of two hundred datasets of histogram data of known cluster structure. We then evaluated the quality of each algorithm using an index for the evaluation of the agreement and one for the accuracy between the obtained partitions and the initial one. The real data analysis was performed using data from a set of meteorological stations in the People's Republic of China. In the following we present how the experiment have been set up and what indices have been used for assessing the performances of the algorithms.

\subsection{Synthetic dataset}
In this section, we will describe the performances of the two Dynamic Clustering methods based on adaptive squared Wasserstein distances when the data hold structures of (controlled) variability. In particular, we aim to show that GC-AWD has the best performance when the data present a different dispersion for each component of each histogram variable. Further, considering a different dispersion (for each component of the variables) for each cluster, we aim to show CDC-AWD gives better results when the data present a different dispersion for each component of the histogram variables for each cluster.\\
We cannot compare the use of the Wasserstein distance with other distances for distributions, because it is not assured the decomposition in terms of mean and of dispersion for the other distances presented in the literature (see \cite{VerIRP08did}).\\
To this end, we set up two Monte Carlo experiments. Each experiment consisted of 100 generations of 150 synthetic objects (3 clusters of 50 objects), described by two histogram variables. Each experiment was initialized by choosing, for each cluster and each variable, four parameters (mean, standard deviation, skewness and kurtosis), obtaining $3(clusters)\times 2(variables)$ baseline sets of four parameters. Starting from the baseline parameters, we assigned a standard deviation for each parameter and repeated the following steps 100 times:
      \begin{enumerate}
  \item for each cluster and each variable, we generated 50 sets of parameters, adding a random error (consistent with the standard deviation of the parameter) to each baseline set of four parameters on the basis of the object belonging to the class, obtaining $3(clusters)\times 2(variables)\times 50 (objects)$ sets of four parameters;
  \item  for each object and each variable, we generated 1,000 random numbers using a Pearson parametric distribution
      (for further details see \citet[Pg. 15, Eqn. 12.33]{John94}) based on the sets of parameters;
  \item for each object and each variable, we computed a histogram using the algorithm presented in \cite{IrpinoR07}, obtaining $3(clusters)\times 2(variables)\times 50 (objects)$ histograms;
      \item each clustering method (the standard Dynamic Clustering and the four adaptive distanced based ones) was randomly initialized 50 times and  we chose the best final clustering result (i.e., the one with lower criterion);
          \item for each best solution we computed the Corrected Rand Index \cite{HUBAR85} to indicate the quality of the result with respect to the initial labels of the data.
\end{enumerate}

To measure the quality of the results furnished by the dynamic clustering algorithm considering different adaptive distances,  an external validity index can be adopted. Indeed, in this case the data were labeled on the basis of the generator functions used to set up the experiments. In this paper, we use the Corrected Rand (CR) index, defined by \citet{HUBAR85} for comparing two partitions. $CR$ takes its values in $[-1,1]$ interval, where the value $1$ indicates perfect agreement between partitions; whereas values close to zero (or negatives) correspond to cluster agreement found by chance.
Further we use the accuracy index as the percentage of correct classified objects.

At the end of each experiment, we report the main statistics for the CR index and for the accuracy.

\paragraph{Experiment 1}
Considering the variability of a \emph{histogram} variable as a combination of the variability of the moments associated with each histogram description of each individual, in the first experiment we generated data in order to obtain two histogram variables that locally (for each cluster) present the same \emph{histogram} variability (i.e., each \emph{histogram} variable for each cluster has the same variability), while globally, the two histogram variables present different variability. In order to obtain the datasets, we fixed the baseline parameter space sampling from $3 (clusters)\;\times 2(variables)\; \times 4\; (parameters)$ Normal distributions, as described in Table \ref{tab:exp2}.\\

\begin{table}[htbp]
  \centering
  \caption{Baseline settings for Experiment 1: each couple are respectively the mean and standard deviation of the sampled Normal distributions used for extracting the parameters of a Pearson's family distribution that has been sampled for the extraction of values summarized by histogram data.}
    \resizebox{\textwidth}{!}{

    \begin {tabular}{cc}
    \begin{tabular}{rcccc}

          & \multicolumn{ 4}{c}{Variable 1: space of parameters}       \\
          & Mean  & St.dev & Skewness & Kurtosis \\
          \hline
    Cluster 1 & (-4.8, 6) & (12, 1.2) & (-0.05, 0.1) & ( 3.10, 0.1) \\
    Cluster 2 & (-4.8, 6) & ( 9, 1.2) & ( 0.00, 0.1) & ( 3.00, 0.1) \\
    Cluster 3 & (10.0, 6) & ( 6, 1.2) & ( 0.10, 0.1) & ( 2.95, 0.1) \\
    \hline
    \end{tabular}&
    \begin{tabular}{cccc}
          & \multicolumn{ 3}{c}{Variable 2: space of parameters}       \\
           Mean  & St.dev & Skewness & Kurtosis \\
    \hline
     ( 17, 12) & (6.0, 0.6) & ( 0.1, 0.1) & (2.95, 0.1) \\
     (-17, 12) & (4.6, 0.6) & ( 0.0, 0.1) & (3.00, 0.1) \\
     (  0, 12) & (3.3, 0.6) & (-0.1, 0.1) & (3.10, 0.1) \\
    \hline
    \end{tabular}
    \end{tabular}
}
  \label{tab:exp2}
\end{table}
In order to have a visualization of the space of the parameters, Fig.\ref{exp2} shows the ellipses for each cluster and for each parameter.\\
\begin{figure}[htbp]
\centering
  \includegraphics[width=0.5\textwidth]{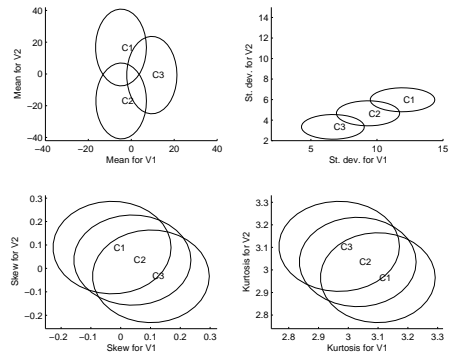}\\
  \caption{Experiment 1: 95\% ellipses for the bivariate distributions of the parameters of the distributions defined in Table \ref{tab:exp2}.}\label{exp2}
\end{figure}
After the generation of 100 datasets according the the baseline parameters, we computed the CR index as well as the accuracy for each algorithm. The means and the standard deviations of the agreement indices are summarized in Table \ref{tab:cr2}.
\begin{table}[htbp]
  \centering
  \caption{Mean of the best 100 CR indices  and accuracies for Experiment 1}
    \begin{tabular}{lccc}

          & STANDARD  & GC-AWD & CDC-AWD \\
          \hline
    Mean Best CR (std)& 0.4979 (0.0097) &  \textbf{0.6596 (0.0072)}  & 0.6292 (0.0094)\\
Mean Best Accuracy (std) &0.7867 (0.0410)&\textbf{0.8667 (0.0340)}&\textbf{0.8667(0.0340)}\\
    \hline
    \end{tabular}
  \label{tab:cr2}
\end{table}
The results of Experiment 1 show that the Dynamic Clustering Algorithms based on adaptive distances outperformed the classic Dynamic Clustering. In this case, it can be seen that GC-AWD and CDC-AWD had similar performances, but the algorithm  GC-AWD allowed the best results in terms of the mean of CR index .

\paragraph{Experiment 2}
The second experiment was based on the generation of histogram data in order to obtain two histogram variables where, for each \emph{histogram} variable and for each cluster there is different variability, while globally, the two variables presented a similar variability. In order to obtain the 100 datasets, we fixed the baseline parameter space sampling from $3 (clusters)\;\times 2(variables)\; \times 4\; (parameters)$ Normal distributions, as described in Table \ref{tab:exp4}.\\
\begin{table}[htbp]
  \centering
  \caption{Baseline settings for Experiment 2:  each couple are respectively the mean and standard deviation of the sampled Normal distributions used for extracting the parameters of a Pearson's family distribution that has been sampled for the extraction of values summarized by histogram data.}
     \resizebox{\textwidth}{!}{

    \begin {tabular}{cc}
    \begin{tabular}{rcccc}

          & \multicolumn{ 4}{c}{Variable 1: space of parameters}       \\
          & Mean  & St.dev & Skewness & Kurtosis \\
          \hline
    Cluster 1 & ( 0.0, 0.8) & (3.6, 0.3) & (-0.04, 0.01) & (2.90, 0.03) \\
    Cluster 2 & (-0.5, 1.6) & (2.7, 0.2) & ( 0.03, 0.01) & (3.05, 0.03) \\
    Cluster 3 & ( 2.8, 2.4) & (1.8, 0.1) & ( 0.10, 0.01) & (3.20, 0.03) \\
    \hline
    \end{tabular}&
    \begin{tabular}{cccc}
          & \multicolumn{ 3}{c}{Variable 2: space of parameters}       \\
           Mean  & St.dev & Skewness & Kurtosis \\
    \hline
    ( 0.0, 2.3) & (4.1, 0.1) & ( 0.10, 0.01) & (3.20, 0.03) \\
    (-3.0, 1.6) & (3.4, 0.2) & ( 0.03, 0.01) & (3.05, 0.03) \\
    ( 1.1, 0.8) & (2.8, 0.3) & (-0.03, 0.01) & (2.90, 0.03) \\
    \hline
    \end{tabular}
    \end{tabular}
}
  \label{tab:exp4}
\end{table}
In order to have a visualization of the space of the parameters, Fig.\ref{exp4} shows the ellipses for each cluster and for each parameter.\\
\begin{figure}[htbp]
  \centering
  \includegraphics[width=0.5\textwidth]{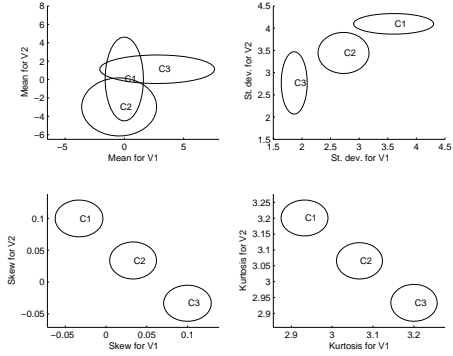}\\
  \caption{Experiment 4: 95\% ellipses for the bivariate distributions of the parameters of the distributions defined in Table \ref{tab:exp4}.}\label{exp4}
\end{figure}
After the generation of 100 datasets according to the baseline parameters, we computed the CR index as well as the accuracy for each algorithm. The means and the standard deviations of the agreement indices are summarized in Table \ref{tab:cr4}.

\begin{table}[htbp]
  \centering
  \caption{Mean of the best 100 CR indices  and accuracies for Experiment 2}
    \begin{tabular}{lccc}
          & STANDARD  & GC-AWD & CDC-AWD \\
          \hline
    Mean Best CR (std)& 0.4255 (0.0064) & 0.4248 (0.0063)& \textbf{0.5376 (0.0095)}\\
Mean Best Accuracy (std) &0.6867 (0.0464)&0.6867(0.0464)&\textbf{0.7133 (0.0452)}\\
    \hline
    \end{tabular}
  \label{tab:cr4}
\end{table}
The second experiment shows that the Dynamic Clustering Algorithms based on adaptive distances based on the schemas of CDC-AWD outperformed the classic Dynamic Clustering. It can be seen that CDC-AWD allowed the best results in terms of the mean CR index  and accuracy, while GC-AWD  did not outperform the STANDARD algorithm.\\
We have empirical evidence that each of the two schemas of adaptive distances allows better clustering results, when considering the two structures of dispersion listed at the beginning of the section.

\subsection{A real dataset: Climatic data from China}
In this section, we use dynamic clustering based on adaptive distances on a dataset where the descriptors are the distributions of the mean monthly temperature, the pressure, the relative humidity, the wind speed and the total monthly precipitations recorded in 60 meteorological stations in the People's Republic of China\footnote{Dataset URL: \texttt{http://dss.ucar.edu/datasets/ds578.5/}}, recorded from 1840 to 1988. For the purposes of this paper, we have considered the distributions of the variables for January (the coldest month) and July (the hottest month), so our initial data is a $60 \times 10$ matrix, where the generic $(i,j)$ cell contains the histogram of the values for the $j^{th}$ variable of the $i^{th}$ meteorological station, defined by means of the algorithm proposed by \citet{IrpinoR07}. Table \ref{tab:basestat} describes the variables and the main statistics related to the global barycenter for each variable (${\bar y}$ and $s$ are, respectively, the mean and the standard deviation of the histogram barycenter, while $TSS$ is computed according the formulation of the Total Sum of Squares proposed by \citet{VeIR08} as a measure of variability of the histogram variable).
\\
\begin{table}[htbp]
  \centering
  \begin{tabular}{rlrrr}
    \hline
    $\sharp$& Variable & ${\bar y}_{j}$ & $s_{j}$ & $TSS_j$\\
    \hline
$Y_1$&Mean Relative Humidity (percent) Jan& 67.9 & 7.0& 127.9 \\
$Y_2$&Mean Relative Humidity (percent) July& 73.9& 4.5&114.2  \\
$Y_3$&Mean Station Pressure (mb) Jan& 968.3& 3.6& 5864.7\\
$Y_4$&Mean Station Pressure (mb) July& 951.1& 3.0&5084.4\\
$Y_5$&Mean Temperature (Cel.$\times10$) Jan& -12.1& 17.3& 114.8\\
$Y_6$&Mean Temperature (Cel.$\times10$) July& 252.3& 10.5& 113.5 \\
$Y_7$&Mean Wind Speed (m/s) Jan& 2.3& 0.6& 1.1\\
$Y_8$&Mean Wind Speed (m/s) July& 2.3& 0.5& 0.6\\
$Y_9$&Total Precipitation (mm) Jan& 18.2& 14.3& 519.6\\
$Y_{10}$&Total Precipitation (mm) July& 144.6& 80.8& 499.9\\
    \hline
  \end{tabular}
  \caption{Basic statistics of the histogram variables: ${\bar y}_j$ and $s_j$ are the mean and the standard deviation of the barycenter histogram for  the $j^{th}$ variable, while $TSS_j$ is the Total Sum of Squares for the $j^{th}$ variable.}\label{tab:basestat}
\end{table}
Table \ref{tab:data3} shows the main characteristics of a subset of the observed stations. The mean, standard deviation, the skewness and the kurtosis are reported for each histogram description. The last two measures correspond to the third and fourth standardized moments. Table \ref{tab:data3} shows also the spatial coordinates of the stations (longitude, latitude and elevation in meters). This information is not relevant to the analysis except for the interpretation of the obtained cluster, indeed the spatial coordinates do not play an active role in the analysis.
\begin{table}
  \centering
  \caption{Dataset from 60 climatic stations in China: main characteristics for each station.}
    \resizebox{\textwidth}{!}{
     \begin{tabular}{l}
     \begin{tabular}{|rl|rrr|rrrr|rrrr|rrrr|}
     \hline
          &       &       &       &       & \multicolumn{ 4}{|c|}{$Y_1$}    & \multicolumn{ 4}{|c|}{$Y_2$}    & \multicolumn{ 4}{|c|}{$Y_3$}    \\
    ID    & St.Name & {\bf Long.} & {\bf Lat.} & {\bf Elev.}  & Mean  & St.dev. & Skew. & Kurt. & Mean  & St.dev. & Skew. & Kurt. & Mean  & St.dev. & Skew. & Kurt. \\
    \hline
    1     & Hailaer & 119.75 & 49.22 & 612.80 & 781.77 & 59.34 & 1.26  & 5.78  & 708.21 & 54.58 & -0.88 & 3.28  & 9,477.82 & 25.02 & -0.72 & 3.77 \\
    2     & NenJiang & 125.23 & 49.17 & 242.20 & 740.16 & 48.88 & 0.31  & 3.24  & 779.10 & 38.95 & -0.32 & 1.89  & 9,920.16 & 28.91 & -0.13 & 2.42 \\
    3     & BoKeTu & 121.92 & 48.77 & 739.40 & 690.43 & 46.44 & 1.22  & 5.68  & 782.71 & 42.46 & -0.51 & 2.32  & 9,318.44 & 39.43 & -0.04 & 2.13 \\
    4     & QiQiHaEr & 123.92 & 47.38 & 145.90 & 688.03 & 67.16 & 0.11  & 3.08  & 727.01 & 57.91 & -0.53 & 3.16  & 10,051.41 & 26.14 & -0.10 & 2.33 \\
    $\cdots$&$\cdots$&$\cdots$&$\cdots$&$\cdots$&$\cdots$&$\cdots$&$\cdots$&$\cdots$&$\cdots$&$\cdots$&$\cdots$& $\cdots$&$\cdots$&$\cdots$&$\cdots$&$\cdots$ \\
    59    & ZhanJiang & 110.40 & 21.22 & 25.30 & 783.32 & 64.65 & -0.98 & 3.59  & 808.81 & 20.74 & -0.06 & 2.90  & 10,166.05 & 20.78 & -0.34 & 2.95 \\
    60    & HaiKou & 110.35 & 20.03 & 14.10 & 851.33 & 41.81 & -0.62 & 3.23  & 819.64 & 28.38 & -0.48 & 3.97  & 10,175.49 & 26.03 & 0.66  & 4.59 \\
    \hline
    \end{tabular}
    \\
    \\
  \begin{tabular}{|rl|rrrr|rrrr|rrrr|}
  \hline
          &       & \multicolumn{ 4}{|c|}{$Y_4$}    & \multicolumn{ 4}{|c|}{$Y_5$}    & \multicolumn{ 4}{|c|}{$Y_6$}    \\
    ID    & STNAME & Mean  & St.dev. & Skew. & Kurt. & Mean  & St.dev. & Skew. & Kurt. & Mean  & St.dev. & Skew. & Kurt. \\
    \hline
    1     & Hailaer & 9,339.85 & 16.68 & -0.39 & 3.67  & -275.31 & 32.40 & -0.86 & 3.72  & 200.71 & 13.60 & 0.96  & 4.47 \\
    2     & NenJiang & 9,756.11 & 18.44 & 0.50  & 2.66  & -254.96 & 26.44 & -0.11 & 2.42  & 206.27 & 9.76  & 0.44  & 2.70 \\
    3     & BoKeTu & 9,224.33 & 29.67 & 0.71  & 2.56  & -216.13 & 26.50 & -0.69 & 3.77  & 179.46 & 10.87 & 0.66  & 4.04 \\
    4     & QiQiHaEr & 9,861.43 & 15.68 & 0.08  & 1.95  & -197.70 & 23.36 & -0.49 & 2.64  & 227.18 & 10.68 & 0.06  & 2.69 \\
    $\cdots$&$\cdots$&$\cdots$&$\cdots$&$\cdots$&$\cdots$&$\cdots$&$\cdots$&$\cdots$&$\cdots$&$\cdots$&$\cdots$& $\cdots$&$\cdots$\\
    58    & NanNing & 9,946.66 & 19.69 & -1.41 & 4.35  & 129.52 & 18.88 & -0.69 & 3.18  & 283.90 & 5.96  & -0.68 & 4.29 \\
    59    & ZhanJiang & 10,012.91 & 13.96 & 0.00  & 3.34  & 157.52 & 16.68 & -0.45 & 2.66  & 288.28 & 5.42  & 0.12  & 2.48 \\
    60    & HaiKou & 10,028.93 & 18.09 & 0.44  & 3.63  & 174.09 & 17.78 & 0.02  & 2.55  & 284.05 & 5.49  & -0.27 & 3.11 \\
    \hline
    \end{tabular}
    \\
    \\
    \begin{tabular}{|rl|rrrr|rrrr|rrrr|rrrr|}
    \hline
          &       & \multicolumn{ 4}{|c|}{$Y_7$}    & \multicolumn{ 4}{|c|}{$Y_8$}    & \multicolumn{ 4}{|c|}{$Y_9$}    & \multicolumn{ 4}{|c|}{$Y_{10}$} \\
    ID    & STNAME & Mean  & St.dev. & Skew. & Kurt. & Mean  & St.dev. & Skew. & Kurt. & Mean  & St.dev. & Skew. & Kurt. & Mean  & St.dev. & Skew. & Kurt. \\
    \hline
    1     & Hailaer & 18.80 & 7.60  & 0.77  & 3.47  & 27.84 & 6.51  & 1.25  & 6.66  & 34.79 & 24.92 & 1.27  & 4.55  & 886.70 & 424.09 & 0.43  & 3.63 \\
    2     & NenJiang & 15.99 & 6.78  & 1.25  & 3.67  & 26.97 & 7.37  & 0.64  & 2.72  & 31.67 & 22.05 & 1.19  & 4.34  & 1,397.45 & 683.74 & 0.73  & 4.42 \\
    3     & BoKeTu & 33.97 & 7.44  & 0.01  & 3.14  & 20.44 & 4.32  & 0.46  & 2.94  & 25.41 & 23.90 & 1.06  & 3.40  & 1,362.48 & 605.98 & 0.86  & 3.32 \\
    4     & QiQiHaEr & 28.32 & 6.89  & 0.36  & 2.52  & 29.71 & 5.09  & 0.03  & 2.64  & 17.49 & 17.74 & 1.17  & 4.52  & 1,321.97 & 753.20 & 0.77  & 3.42 \\
    $\cdots$&$\cdots$&$\cdots$&$\cdots$&$\cdots$&$\cdots$&$\cdots$&$\cdots$&$\cdots$&$\cdots$&$\cdots$&$\cdots$& $\cdots$&$\cdots$&$\cdots$&$\cdots$&$\cdots$&$\cdots$\\
    58    & NanNing & 16.83 & 4.33  & -0.04 & 2.57  & 20.37 & 4.04  & -0.30 & 2.30  & 317.73 & 293.68 & 1.59  & 6.26  & 2,003.83 & 872.76 & 0.36  & 2.77 \\
    59    & ZhanJiang & 31.01 & 9.34  & 1.27  & 5.62  & 28.86 & 5.67  & 0.23  & 2.70  & 221.59 & 248.06 & 1.49  & 4.87  & 2,080.67 & 1,146.92 & 0.56  & 3.08 \\
    60    & HaiKou & 32.65 & 8.63  & 0.62  & 2.53  & 25.84 & 6.62  & -0.15 & 2.14  & 212.79 & 185.36 & 0.77  & 2.45  & 1,899.15 & 1,029.41 & 1.46  & 6.26 \\
    \hline
    \end{tabular}
    \end{tabular}
    }
  \label{tab:data3}
\end{table}
In this case, we use clustering in an explorative fashion. For the sake of brevity, we show the results of only one method.
The choices for the method and number of clusters were made according to the maximum value observed for the \citet{Cali75} index (CH). The CH index is a validity index that is generally used for the determination of the number of clusters, and can be viewed as a Pseudo-F index. If $K$ is the number of clusters of a partition of a set of $n$ individuals, $WSS(K)$ the Within Sum of Squares and $BSS(K)$ the Between Sum of Squares, the $CH$ index is computed as follows:
\begin{equation}\label{CHindex}
    CH(K)=\frac{BSS(K)/(K-1)}{WSS(K)/(n-K)}.
\end{equation}
In this case, we executed 100 initializations for each clustering algorithm and a number of clusters of the partition going from $2$ to $10$. Among the three algorithms, the method that allowed a good behavior for the $CH$ index was CDC-AWD. Indeed, as it is shown in Fig. \ref{fig:CHindex}, it reaches a maximum when we choose a partition in $K=8$ clusters, while the other methods do not present an absolute maximum suggesting a scarce cluster structure of the data. For $K=8$, CDC-AWD registered an high value for the $QPI$ index ($0.928$), while the dynamic clustering based on non-adaptive distances (STANDARD) had a $QPI$ equal to $0.873$, and  GC-AWD  had $QPI$'s respectively equal to  $0.825$.
\begin{figure}[htbp]
  \centerline{
    \includegraphics[width=.5\textwidth]{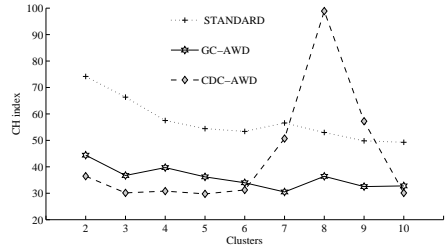}}
  \caption{CH indices for the three algorithms and for $K=2,\ldots,10$.}
  \label{fig:CHindex}
\end{figure}
Table \ref{tab:China_Lambda} shows the final set of $\lambda$ weights associated with the components (the mean and the dispersion)of the distance in each cluster according to the CDC-AWD algorithm. The $\lambda$'s can be considered as normalization factors for the component of the variables; indeed, we can see a high value for the weights associated with the variables $Y_7$ and $Y_8$ (the Wind Speed in January and in July) that have, in general, a lower $TSS$. Because we used the CDC-AWD algorithm, the weights take into account the within cluster variability of the components of the variables according to each cluster, so, in this case, we must read the weights cluster by cluster. Indeed, the constraint of product equal to 1 related to the weights computed according to Eqs. \ref{lambda4_1} and \ref{lambda4_2} are referred to each cluster and cannot generally be compared across the different clusters.
\begin{table}[htbp]
  \centering
  \caption{Weights generated by CDC-AWD on China data, $K=8$. $n_k$'s are the cardinality of the obtained clusters.}
 \resizebox{\textwidth}{!} {
  \begin{tabular}{rrrrrrrrrrrr}
          &       & \multicolumn{ 2}{c}{$Y_1$} & \multicolumn{ 2}{c}{$Y_2$} & \multicolumn{ 2}{c}{$Y_3$} & \multicolumn{ 2}{c}{$Y_4$} & \multicolumn{ 2}{c}{$Y_5$} \\
          & $n_k$     & ${\bar y}$ & $Disp$ & ${\bar y}$ & $Disp$ & ${\bar y}$ & $Disp$ & ${\bar y}$ & $Disp$ & ${\bar y}$ & $Disp$ \\
\hline
    Cluster 1 & 10    & 0.866 & 0.558 & 1.230 & 1.177 & 0.080 & 0.465 & 0.067 & 0.582 & 0.904 & 14.628 \\
    Cluster 2 & 8     & 1.251 & 0.234 & 1.649 & 0.325 & 0.028 & 1.636 & 0.035 & 1.684 & 0.311 & 1.696 \\
    Cluster 3 & 2     & 0.057 & 0.781 & 0.050 & 0.114 & 0.593 & 0.710 & 528.519 & 0.476 & 0.091 & 10.424 \\
    Cluster 4 & 11    & 0.329 & 0.214 & 0.480 & 0.222 & 0.794 & 1.784 & 1.037 & 1.394 & 1.115 & 6.680 \\
    Cluster 5 & 3     & 0.121 & 0.183 & 0.095 & 1.253 & 0.189 & 0.038 & 0.169 & 0.048 & 2.149 & 7.179 \\
    Cluster 6 & 6     & 0.199 & 1.292 & 0.654 & 1.337 & 0.009 & 0.006 & 0.010 & 0.005 & 1.489 & 25.620 \\
    Cluster 7 & 8     & 1.561 & 0.161 & 0.226 & 0.397 & 0.014 & 1.868 & 0.017 & 0.191 & 1.236 & 5.317 \\
    Cluster 8 & 12    & 0.691 & 0.811 & 1.346 & 1.101 & 0.548 & 0.514 & 1.661 & 0.467 & 0.280 & 21.020 \\
\hline
\hline
          &       & \multicolumn{ 2}{c}{$Y_6$} & \multicolumn{ 2}{c}{$Y_7$} & \multicolumn{ 2}{c}{$Y_8$} & \multicolumn{ 2}{c}{$Y_9$} & \multicolumn{ 2}{c}{$Y_{10}$} \\
          & $n_k$     & ${\bar y}$ & $Disp$ & ${\bar y}$ & $Disp$ & ${\bar y}$ & $Disp$ & ${\bar y}$ & $Disp$ & ${\bar y}$ & $Disp$ \\
\hline
    Cluster 1 & 10    & 139.669 & 22.629 & 38.096 & 25.118 & 77.594 & 55.284 & 0.090 & 0.010 & 0.005 & 0.001 \\
    Cluster 2 & 8     & 2.929 & 11.920 & 24.585 & 12.705 & 66.030 & 15.128 & 14.789 & 0.420 & 0.023 & 0.003 \\
    Cluster 3 & 2     & 119.150 & 7.413 & 18.507 & 219.007 & 22.134 & 34.659 & 2.106 & 0.234 & 0.000 & 0.000 \\
    Cluster 4 & 11    & 6.248 & 11.517 & 10.351 & 29.491 & 16.268 & 53.147 & 0.423 & 0.049 & 0.016 & 0.001 \\
    Cluster 5 & 3     & 2.572 & 17.131 & 96.803 & 103.463 & 105.639 & 26.841 & 0.261 & 0.095 & 0.185 & 0.073 \\
    Cluster 6 & 6     & 6.759 & 126.531 & 343.066 & 93.389 & 632.098 & 205.222 & 1.147 & 0.135 & 0.032 & 0.002 \\
    Cluster 7 & 8     & 18.122 & 11.068 & 73.259 & 28.923 & 112.579 & 36.336 & 6.912 & 0.316 & 0.009 & 0.002 \\
    Cluster 8 & 12    & 24.848 & 12.611 & 21.428 & 23.214 & 24.027 & 16.495 & 0.038 & 0.020 & 0.009 & 0.002 \\
\hline
    \end{tabular}
    }
  \label{tab:China_Lambda}
\end{table}
 Figure \ref{fig:ch_clu} is the map of the 60 stations. Each station is marked with a symbol that represents the cluster to which it belongs. Note that the clusters represent geographical areas that are consistent with the climate zones of China.\\
\begin{figure}
  \includegraphics[width=0.5\textwidth]{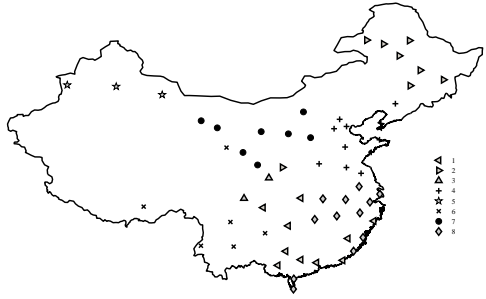}\\
  \caption{Map of the 60 stations clustered into $K=8$ classes using CDC-AWD. Each station is marked with a symbol that represents the cluster to which it belongs. }\label{fig:ch_clu}
\end{figure}
In order to comment on the clustering results and considering the additive properties of the $WSS$, the $BSS$ , and (as consequence) $TSS$  with respect to the components of the variables, the variables and the clusters, we present different versions of the $QPI$ in Table \ref{tab:QPIs}, as proposed by \citet{deca06b}:
\begin{itemize}
  \item The $QPI$'s denoted by $\frac{BSS_{{\bar y}_kj}}{TSS_{{\bar y}_kj}}$ and  $\frac{BSS_{Disp_kj}}{TSS_{Disp_kj}}$ for the components and by $\frac{BSS_{kj}}{TSS_{kj}}$ for the variables in each cluster, allow us to measure the homogeneity of a cluster with respect to the components and the variables: the closer they are to 1, the more likely the cluster is to contain similar objects.
  \item The $QPI$'s denoted by $\frac{BSS_{{\bar y}_k}}{TSS_{{\bar y}_k}}$ and  $\frac{BSS_{Disp_k}}{TSS_{Disp_k}}$ for the components and by $\frac{BSS_{k}}{TSS_{k}}$ for all the variables in each cluster, allow the measurement of the homogeneity of a cluster with respect to all the components and all the variables: the closer they are to 1, the more likely the cluster is to contain similar objects for all the components or all the variables.
  \item The $QPI$'s related to each component $\left(\frac{BSS_{{\bar y}_j}}{TSS_{{\bar y}_j}}\;\mathrm{and}\; \frac{BSS_{Disp_j}}{TSS_{Disp_j}}\right)$ and each variable $\left(\frac{BSS_j}{TSS_j}\right)$ for all the clusters,  allow us to understand the contribution of each component or of each variable to the cluster separation (The closer the QPI is to 1, the more the clusters are separated).
  \item The global $QPI$'s related to the components $\left(\frac{BSS_{{\bar y}}}{TSS_{{\bar y}}}\;\mathrm{and}\; \frac{BSS_{Disp}}{TSS_{Disp}}\right)$ and the $QPI$ related to all the variables $\left(\frac{B}{T}\right)$ allow us to evaluate the global quality of the clustering results for each component and each variable.
\end{itemize}
In general, the CDC-AWD algorithm reaches a $QPI=\frac{BSS}{TSS}=0.928$ for a number of clusters equal to $K=8$. Specifically, the mean component of the variables allows higher homogeneity within clusters $\left(\frac{BSS_{\bar y}}{TSS_{\bar y}}=0.932\right)$, while the dispersion component has a medium effect $\left(\frac{BSS_{Disp}}{TSS_{Disp}}=0.553\right)$.
Considering the second group of $QPI$ indices from the results presented in Table \ref{tab:QPIs}, cluster 3 contains objects that are very similar. For dispersion structure,  cluster 2 is slightly more globally heterogenous. Considering the effect of each component in clustering the current dataset,  the two variables that allow better homogeneity in the cluster are $Y_3$ and $Y_4$ (Mean Station Pressure in January and in July), while the worst results are observed for $Y_7$ (Mean Wind Speed in January) and $Y_2$ (Mean Relative Humidity in July).\\
\begin{table}[htbp]
  \centering
  \caption{$QPI$s resulting from CDC-AWD, which has been used for partitioning China's stations into $K=8$ clusters. The $QPI$s are detailed for each component, for each variable and for each cluster.}
  \resizebox{0.9\textwidth}{!}{
     \begin{tabular}{|lrrrrrrrrrr|r|}
    \multicolumn{ 12}{c}{$QPI_{\bar y}$}                                                                 \\
    \hline
    $\frac{BSS_{{\bar y}_kj}}{TSS_{{\bar y}_kj}}$ & $Y_1$ & $Y_2$ & $Y_3$ & $Y_4$ & $Y_5$ & $Y_6$ & $Y_7$ & $Y_8$ & $Y_9$ & $Y_{10}$ & $\frac{BSS_{{\bar y}_k}}{TSS_{{\bar y}_k}}$   \\
   \hline
    Cluster 1 & 0.860 & 0.548 & 0.072 & 0.933 & 0.929 & 0.920 & 0.092 & 0.132 & 0.908 & 0.800 & {\bf 0.857}   \\
    Cluster 2 & 0.469 & 0.347 & 0.604 & 0.365 & 0.832 & 0.820 & 0.435 & 0.720 & 0.016 & 0.596 & {\bf 0.641}   \\
    Cluster 3 & 0.960 & 0.838 & 1.000 & 1.000 & 0.924 & 0.999 & 0.986 & 0.978 & 0.986 & 0.854 & {\bf 0.999}   \\
    Cluster 4 & 0.564 & 0.075 & 0.837 & 0.995 & 0.129 & 0.486 & 0.606 & 0.524 & 0.407 & 0.899 & {\bf 0.957}   \\
    Cluster 5 & 0.232 & 0.921 & 0.992 & 0.926 & 0.962 & 0.688 & 0.433 & 0.881 & 0.316 & 0.993 & {\bf 0.969}   \\
    Cluster 6 & 0.304 & 0.060 & 0.868 & 0.781 & 0.430 & 0.860 & 0.040 & 0.359 & 0.174 & 0.750 & {\bf 0.681}   \\
    Cluster 7 & 0.952 & 0.844 & 0.956 & 0.902 & 0.827 & 0.973 & 0.543 & 0.137 & 0.770 & 0.423 & {\bf 0.911}   \\
    Cluster 8 & 0.881 & 0.435 & 0.825 & 0.998 & 0.763 & 0.840 & 0.773 & 0.730 & 0.889 & 0.792 & {\bf 0.985}   \\
    \hline
    $\frac{BSS_{{\bar y}_j}}{TSS_{{\bar y}_j}}$ & {\bf 0.781} & {\bf 0.525} & {\bf 0.923} & {\bf 0.990} & {\bf 0.790} & {\bf 0.897} & {\bf 0.465} & {\bf 0.564} & {\bf 0.694} & {\bf 0.879} & {\bf 0.932} \\
    & & & & & & & & & & & $\frac{BSS_{{\bar y}}}{TSS_{{\bar y}}}$ \\
 \hline

    \multicolumn{ 12}{c}{$QPI_{Disp}$}                                                              \\
    \hline
    $\frac{BSS_{Disp_kj}}{TSS_{Disp_kj}}$ & $Y_1$ & $Y_2$ & $Y_3$ & $Y_4$ & $Y_5$ & $Y_6$ & $Y_7$ & $Y_8$ & $Y_9$ &  $Y_{10}$ & $\frac{BSS_{{\bar y}_k}}{TSS_{Disp_k}}$   \\
    \hline
    Cluster 1 & 0.123 & 0.204 & 0.110 & 0.091 & 0.002 & 0.197 & 0.052 & 0.075 & 0.783 & 0.792 & {\bf 0.457}   \\
    Cluster 2 & 0.297 & 0.369 & 0.359 & 0.066 & 0.712 & 0.612 & 0.522 & 0.362 & 0.804 & 0.769 & {\bf 0.596}   \\
    Cluster 3 & 0.959 & 0.379 & 0.524 & 0.641 & 0.943 & 0.385 & 0.984 & 0.898 & 0.398 & 0.833 & {\bf 0.923}   \\
    Cluster 4 & 0.684 & 0.096 & 0.147 & 0.192 & 0.022 & 0.393 & 0.566 & 0.245 & 0.459 & 0.883 & {\bf 0.572}   \\
    Cluster 5 & 0.244 & 0.613 & 0.318 & 0.251 & 0.884 & 0.664 & 0.245 & 0.338 & 0.176 & 0.960 & {\bf 0.789}   \\
    Cluster 6 & 0.095 & 0.059 & 0.151 & 0.152 & 0.131 & 0.302 & 0.135 & 0.173 & 0.242 & 0.599 & {\bf 0.247}   \\
    Cluster 7 & 0.629 & 0.550 & 0.271 & 0.142 & 0.364 & 0.396 & 0.033 & 0.148 & 0.629 & 0.028 & {\bf 0.398}   \\
    Cluster 8 & 0.301 & 0.103 & 0.039 & 0.012 & 0.136 & 0.147 & 0.449 & 0.281 & 0.915 & 0.889 & {\bf 0.676}   \\
    \hline
    $\frac{BSS_{Disp_j}}{TSS_{Disp_j}}$ & {\bf 0.413} & {\bf 0.258} & {\bf 0.169} & {\bf 0.126} & {\bf 0.448} & {\bf 0.358} & {\bf 0.422} & {\bf 0.235} & {\bf 0.768} & {\bf 0.840} & {\bf 0.553} \\
    & & & & & & & & & & & $\frac{BSS_{Disp}}{TSS_{Disp}}$ \\
    \hline

    \multicolumn{ 12}{c}{$QPI$}                                                                     \\
    \hline
    $\frac{BSS_{kj}}{TSS_{kj}}$ & $Y_1$ & $Y_2$ & $Y_3$ & $Y_4$ & $Y_5$ & $Y_6$ & $Y_7$ & $Y_8$ & $Y_9$ & $Y_{10}$ & $\frac{BSS_{k}}{TSS_{k}}$   \\
    \hline
    Cluster 1 & 0.841 & 0.519 & 0.078 & 0.923 & 0.918 & 0.909 & 0.087 & 0.125 & 0.900 & 0.799 & {\bf 0.841}   \\
    Cluster 2 & 0.465 & 0.348 & 0.599 & 0.358 & 0.830 & 0.817 & 0.439 & 0.715 & 0.126 & 0.605 & {\bf 0.640}   \\
    Cluster 3 & 0.960 & 0.770 & 1.000 & 1.000 & 0.933 & 0.999 & 0.985 & 0.968 & 0.977 & 0.846 & {\bf 0.999}   \\
    Cluster 4 & 0.574 & 0.076 & 0.829 & 0.995 & 0.123 & 0.481 & 0.604 & 0.514 & 0.410 & 0.899 & {\bf 0.955}   \\
    Cluster 5 & 0.234 & 0.910 & 0.990 & 0.914 & 0.958 & 0.685 & 0.410 & 0.864 & 0.298 & 0.992 & {\bf 0.964}   \\
    Cluster 6 & 0.295 & 0.060 & 0.861 & 0.771 & 0.418 & 0.853 & 0.046 & 0.351 & 0.178 & 0.745 & {\bf 0.670}   \\
    Cluster 7 & 0.948 & 0.835 & 0.953 & 0.895 & 0.816 & 0.971 & 0.523 & 0.138 & 0.763 & 0.404 & {\bf 0.905}   \\
    Cluster 8 & 0.865 & 0.404 & 0.803 & 0.998 & 0.737 & 0.820 & 0.753 & 0.704 & 0.893 & 0.814 & {\bf 0.983}   \\
    \hline
    $\frac{BSS_{j}}{TSS_{j}}$ & {\bf 0.770} & {\bf 0.512} & {\bf 0.918} & {\bf 0.989} & {\bf 0.780} & {\bf 0.891} & {\bf 0.462} & {\bf 0.549} & {\bf 0.701} & {\bf 0.877} & {\bf 0.928}\\
    & & & & & & & & & & & $B/T$ \\
    \hline
    \end{tabular}
    }
  \label{tab:QPIs}
\end{table}
Finally, in Figs. \ref{fig:Prov3_4} and \ref{fig:Prov7_8},  we show the prototypes of the obtained clusters for the better (in terms of $QPI$) and for the worst variables in partitioning the 60 stations into $K=8$ clusters.


\begin{figure}
        \centering
        \includegraphics[width=.5\textwidth]{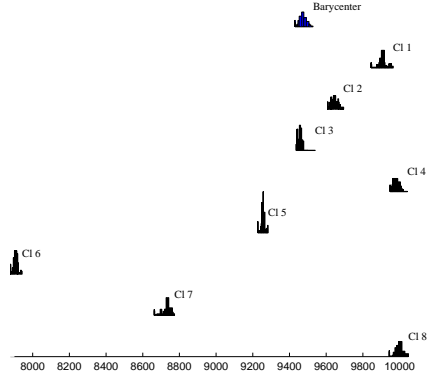}
  \caption{The general prototype and the cluster prototypes for the variable $Y_4$: Mean Station Pressure (mb) in July.}
  \label{fig:Prov3_4}
\end{figure}


\begin{figure}
\centering
    \includegraphics[width=.5\textwidth]{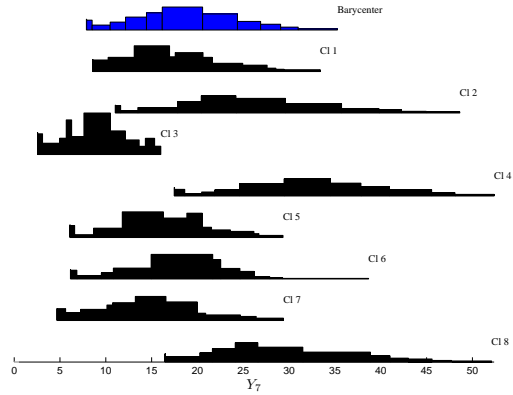}
  \caption{The general prototype and the cluster prototypes for the variable $Y_7$: Mean Wind Speed (m/s) in January.}
  \label{fig:Prov7_8}
\end{figure}


\section{Conclusions}\label{sec:conclu}
In this paper, we presented two new algorithms for the Dynamic Clustering of histogram data based on two adaptive squared Wasserstein distances. The adaptive clustering dynamic algorithm locally optimizes an adequacy criterion that measures the fitting between the clusters and their representatives (the barycenters) based on distances that change at each iteration. The first algorithm (GC-AWD) uses a globally adaptive squared Wasserstein distance for each one of the two (mean and dispersion) components of the quantile functions of histograms. The second algorithm (CDC-AWD) uses a locally adaptive squared Wasserstein distance  for each of the two (mean and dispersion) components of the quantile functions of histograms that changes according to each cluster. The advantages of using such adaptive distances is the ability to identify clusters of different sizes and shapes, while standard DCA (like in k-means case) finds spherical clusters.
Starting from an initial random partition of the objects, the adaptive DCA alternate three steps until convergence when the adequacy criterion reaches a stationary value, which represents a local (within sum of squares) minimum. In the first two steps, the algorithms give the solution for the best prototype of each cluster as well as the solution for the best adaptive distance (locally for each cluster). In the last step, the algorithm gives the solution for the best partition. The convergence as well as the time complexity of the algorithms were addressed.\\
Two experimental evaluations of the proposed methods were presented. The first was performed using two baseline settings for the generation of two hundreds synthetic datasets in order to show the usefulness of each scheme of adaptive Wasserstein distance in identifying a starting class structure in the data.  The second, using a real-word dataset, showed the use of such algorithms in an exploratory fashion. All the algorithms based on adaptive distances were compared with the standard dynamic clustering algorithm (i.e., based on a standard squared Wasserstein distance).\\
The experiment conducted on artificial data showed that the adaptive algorithms outperformed the standard one in terms of accuracy in identifying the initial class structure of the data. The experiment on the real-word dataset also demonstrated that the adaptive algorithms were able to reach a good quality of partition that was generally greater than the standard non-adaptive algorithm.





\bibliographystyle{model1-num-names}
\bibliography{bibliog}

\begin{thebibliography}{28}
\expandafter\ifx\csname natexlab\endcsname\relax\def\natexlab#1{#1}\fi
\providecommand{\bibinfo}[2]{#2}
\ifx\xfnm\relax \def\xfnm[#1]{\unskip,\space#1}\fi
\bibitem[{Bock and Diday(2000)}]{BoDid00}
\bibinfo{author}{H.~H. Bock}, \bibinfo{author}{E.~Diday},
  \bibinfo{title}{Analysis of Symbolic Data. Exploratory Methods for Extracting
  Statistical Information from Complex Data}, \bibinfo{publisher}{Springer},
  \bibinfo{address}{Berlin}, \bibinfo{year}{2000}.
\bibitem[{Billard and Diday(2007)}]{Billard2007}
\bibinfo{author}{L.~Billard}, \bibinfo{author}{E.~Diday},
  \bibinfo{title}{Symbolic Data Analysis: Conceptual Statistics and Data Mining
  (Wiley Series in Computational Statistics)}, \bibinfo{publisher}{John Wiley
  \& Sons}, \bibinfo{year}{2007}.
\bibitem[{Diday and Noirhomme-Fraiture(2008)}]{DiNoi08}
\bibinfo{author}{E.~Diday}, \bibinfo{author}{M.~Noirhomme-Fraiture},
  \bibinfo{title}{Symbolic Data Analysis and the SODAS Software},
  \bibinfo{publisher}{Wiley-Interscience}, \bibinfo{address}{New York, NY,
  USA}, \bibinfo{year}{2008}.
\bibitem[{Irpino and Verde(2006)}]{IrVer06}
\bibinfo{author}{A.~Irpino}, \bibinfo{author}{R.~Verde},
\newblock \bibinfo{title}{A new wasserstein based distance for the hierarchical
  clustering of histogram symbolic data},
\newblock in: \bibinfo{editor}{Batanjeli}, \bibinfo{editor}{Bock},
  \bibinfo{editor}{Ferligoj}, \bibinfo{editor}{Ziberna} (Eds.),
  \bibinfo{booktitle}{Data Science and Classification},
  \bibinfo{publisher}{Springer}, \bibinfo{address}{Berlin},
  \bibinfo{year}{2006}, pp. \bibinfo{pages}{185--192}.
\bibitem[{Verde and Irpino(2006)}]{IRVERLEC06}
\bibinfo{author}{R.~Verde}, \bibinfo{author}{A.~Irpino},
\newblock \bibinfo{title}{Dynamic clustering of histograms using wasserstein
  metric},
\newblock in: \bibinfo{editor}{A.~Rizzi}, \bibinfo{editor}{M.~Vichi} (Eds.),
  \bibinfo{booktitle}{Proceedings in Computational Statistics, COMPSTAT 2006},
  \bibinfo{organization}{Compstat 2006}, \bibinfo{publisher}{Physica Verlag},
  \bibinfo{address}{Heidelberg}, \bibinfo{year}{2006}, pp.
  \bibinfo{pages}{869--876}.
\bibitem[{Verde and Irpino(2008{\natexlab{a}})}]{VeIR08}
\bibinfo{author}{R.~Verde}, \bibinfo{author}{A.~Irpino},
\newblock \bibinfo{title}{Comparing histogram data using a
  mahalanobis-wasserstein distance},
\newblock in: \bibinfo{editor}{P.~Brito} (Ed.), \bibinfo{booktitle}{Proceedings
  in Computational Statistics, COMPSTAT 2008}, \bibinfo{organization}{Compstat
  2008}, \bibinfo{publisher}{Springer Verlag}, \bibinfo{address}{Heidelberg},
  \bibinfo{year}{2008}{\natexlab{a}}, pp. \bibinfo{pages}{77--89}.
\bibitem[{Verde and Irpino(2008{\natexlab{b}})}]{VerIRP08did}
\bibinfo{author}{R.~Verde}, \bibinfo{author}{A.~Irpino},
\newblock \bibinfo{title}{Dynamic clustering of histogram data: using the right
  metric},
\newblock in: \bibinfo{editor}{P.~Brito}, \bibinfo{editor}{P.~Bertrand},
  \bibinfo{editor}{G.~Cucumel}, \bibinfo{editor}{F.~De~Carvalho} (Eds.),
  \bibinfo{booktitle}{Selected contributions in data analysis and
  classification}, \bibinfo{publisher}{Springer}, \bibinfo{address}{Berlin},
  \bibinfo{year}{2008}{\natexlab{b}}, pp. \bibinfo{pages}{123--134}.
\bibitem[{Diday(1971)}]{Did71}
\bibinfo{author}{E.~Diday},
\newblock \bibinfo{title}{La m\'{e}thode des nu\'{e}s dynamique},
\newblock \bibinfo{journal}{Revue de Statistique Appliqu\'{e}e}
  \bibinfo{volume}{19} (\bibinfo{year}{1971}) \bibinfo{pages}{19--34}.
\bibitem[{Diday and Simon(1976)}]{DiSIM76}
\bibinfo{author}{E.~Diday}, \bibinfo{author}{J.~C. Simon},
\newblock \bibinfo{title}{Clustering analysis},
\newblock in: \bibinfo{editor}{K.~Fu} (Ed.), \bibinfo{booktitle}{Digital
  Pattern Classification}, \bibinfo{publisher}{Springer},
  \bibinfo{address}{Berlin}, \bibinfo{year}{1976}, pp. \bibinfo{pages}{47--94}.
\bibitem[{Gibbs and Su(2002)}]{GiSu02}
\bibinfo{author}{A.~L. Gibbs}, \bibinfo{author}{F.~E. Su},
\newblock \bibinfo{title}{On choosing and bounding probability metrics},
\newblock \bibinfo{journal}{Intl. Stat. Rev.} \bibinfo{volume}{7}
  (\bibinfo{year}{2002}) \bibinfo{pages}{419--435}.
\bibitem[{Rubner et~al.(2000)Rubner, Tomasi, and Guibas}]{Rubn00}
\bibinfo{author}{Y.~Rubner}, \bibinfo{author}{C.~Tomasi},
  \bibinfo{author}{L.~J. Guibas},
\newblock \bibinfo{title}{The earth mover's distance as a metric for image
  retrieval},
\newblock \bibinfo{journal}{Int. J. Comput. Vision} \bibinfo{volume}{40}
  (\bibinfo{year}{2000}) \bibinfo{pages}{99--121}.
\bibitem[{Mallows(1972)}]{Mall72}
\bibinfo{author}{C.~L. Mallows},
\newblock \bibinfo{title}{A note on asymptotic joint normality},
\newblock \bibinfo{journal}{Annals of Mathematics Statistics}
  \bibinfo{volume}{43} (\bibinfo{year}{1972}) \bibinfo{pages}{508--515}.
\bibitem[{Levina and Bickel(2001)}]{LeBick01}
\bibinfo{author}{E.~Levina}, \bibinfo{author}{P.~J. Bickel},
\newblock \bibinfo{title}{The earth mover's distance is the mallows distance:
  Some insights from statistics},
\newblock in: \bibinfo{booktitle}{Proceedings of the 8th International
  Conference On Computer Vision}, volume~\bibinfo{volume}{2},
  \bibinfo{publisher}{IEEE Computer Society}, \bibinfo{year}{2001}, pp.
  \bibinfo{pages}{251--256}.
\bibitem[{De~Carvalho and Lechevallier(2009{\natexlab{a}})}]{DECAYVES09}
\bibinfo{author}{F.~A.~T. De~Carvalho}, \bibinfo{author}{Y.~Lechevallier},
\newblock \bibinfo{title}{Partitional clustering algorithms for symbolic
  interval data based on single adaptive distances},
\newblock \bibinfo{journal}{Pattern Recognition} \bibinfo{volume}{42}
  (\bibinfo{year}{2009}{\natexlab{a}}) \bibinfo{pages}{1223--1236}.
\bibitem[{De~Carvalho and Lechevallier(2009{\natexlab{b}})}]{DECLEC09}
\bibinfo{author}{F.~A.~T. De~Carvalho}, \bibinfo{author}{Y.~Lechevallier},
\newblock \bibinfo{title}{Dynamic clustering of interval-valued data based on
  adaptive quadratic distances},
\newblock \bibinfo{journal}{Trans. Sys. Man Cyber. Part A} \bibinfo{volume}{39}
  (\bibinfo{year}{2009}{\natexlab{b}}) \bibinfo{pages}{1295--1306}.
\bibitem[{de~Souza and De~Carvalho(2007)}]{DECSOUZ07}
\bibinfo{author}{R.~M. C.~R. de~Souza}, \bibinfo{author}{F.~A.~T. De~Carvalho},
\newblock \bibinfo{title}{A clustering method for mixed feature-type symbolic
  data using adaptive squared euclidean distances},
\newblock in: \bibinfo{booktitle}{HIS '07: Proceedings of the 7th International
  Conference on Hybrid Intelligent Systems}, \bibinfo{publisher}{IEEE Computer
  Society}, \bibinfo{address}{Washington, DC, USA}, \bibinfo{year}{2007}, pp.
  \bibinfo{pages}{168--173}.
\bibitem[{De~Carvalho and De~Souza(2010)}]{DeCDeS10}
\bibinfo{author}{F.~A.~T. De~Carvalho}, \bibinfo{author}{R.~M. C.~R. De~Souza},
\newblock \bibinfo{title}{Unsupervised pattern recognition models for mixed
  feature--type symbolic data},
\newblock \bibinfo{journal}{Pattern Recognition Letters} \bibinfo{volume}{31}
  (\bibinfo{year}{2010}) \bibinfo{pages}{430--443}.
\bibitem[{Irpino and Romano(2007)}]{IrpinoR07}
\bibinfo{author}{A.~Irpino}, \bibinfo{author}{E.~Romano},
\newblock \bibinfo{title}{Optimal histogram representation of large data sets:
  Fisher vs piecewise linear approximation},
\newblock in: \bibinfo{editor}{M.~Noirhomme-Fraiture},
  \bibinfo{editor}{G.~Venturini} (Eds.), \bibinfo{booktitle}{EGC}, volume
  \bibinfo{volume}{RNTI-E-9} of \textit{\bibinfo{series}{Revue des Nouvelles
  Technologies de l'Information}},
  \bibinfo{publisher}{C{\'e}padu{\`e}s-{\'E}ditions}, \bibinfo{year}{2007}, pp.
  \bibinfo{pages}{99--110}.
\bibitem[{Villani(2003)}]{Villani03}
\bibinfo{author}{C.~Villani}, \bibinfo{title}{Topics in Optimal
  Transportation}, \bibinfo{publisher}{AMS}, \bibinfo{year}{2003}.
\bibitem[{Cuesta-Albertos et~al.(1997)Cuesta-Albertos, Matr\'{a}n, and
  Tuero-Diaz}]{CuALB97}
\bibinfo{author}{J.~A. Cuesta-Albertos}, \bibinfo{author}{C.~Matr\'{a}n},
  \bibinfo{author}{A.~Tuero-Diaz},
\newblock \bibinfo{title}{Optimal transportation plans and convergence in
  distribution},
\newblock \bibinfo{journal}{Journ. of Multiv. An.} \bibinfo{volume}{60}
  (\bibinfo{year}{1997}) \bibinfo{pages}{72--83}.
\bibitem[{Clark and Rae(1984)}]{Clark84}
\bibinfo{author}{R.~G. Clark}, \bibinfo{author}{M.~S. Rae},
\newblock \bibinfo{title}{A class of wasserstein metrics for probability
  distributions},
\newblock \bibinfo{journal}{Michigan Math. J.} \bibinfo{volume}{31 (2)}
  (\bibinfo{year}{1984}) \bibinfo{pages}{231--240}.
\bibitem[{Diday and Govaert(1997)}]{DiGov77}
\bibinfo{author}{E.~Diday}, \bibinfo{author}{G.~Govaert},
\newblock \bibinfo{title}{Classification automatique avec distances
  adaptatives},
\newblock \bibinfo{journal}{R.A.I.R.O. Informatique Computer Science}
  \bibinfo{volume}{11} (\bibinfo{year}{1997}) \bibinfo{pages}{329--349}.
\bibitem[{Celeux et~al.(1989)Celeux, Diday, Govaert, Lechevallier, and
  Ralambondrainy}]{CelAL89}
\bibinfo{author}{G.~Celeux}, \bibinfo{author}{E.~Diday},
  \bibinfo{author}{G.~Govaert}, \bibinfo{author}{Y.~Lechevallier},
  \bibinfo{author}{H.~Ralambondrainy}, \bibinfo{title}{Classification
  Automatique des Donn\'{e}s}, \bibinfo{publisher}{Bordas},
  \bibinfo{address}{Paris}, \bibinfo{year}{1989}.
\bibitem[{De~Carvalho et~al.(2006)De~Carvalho, Brito, and Bock}]{deca06b}
\bibinfo{author}{F.~De~Carvalho}, \bibinfo{author}{P.~Brito},
  \bibinfo{author}{H.~H. Bock},
\newblock \bibinfo{title}{Dynamic clustering for interval data based on l2
  distance},
\newblock \bibinfo{journal}{Computational Statistics} \bibinfo{volume}{2}
  (\bibinfo{year}{2006}) \bibinfo{pages}{231--250}.
\bibitem[{Meila(2005)}]{Meila05}
\bibinfo{author}{M.~Meila},
\newblock \bibinfo{title}{Comparing clusterings: an axiomatic view},
\newblock in: \bibinfo{booktitle}{In ICML '05: Proceedings of the 22nd
  international conference on Machine learning}, \bibinfo{publisher}{ACM
  Press}, \bibinfo{year}{2005}, pp. \bibinfo{pages}{577--584}.
\bibitem[{Johnson et~al.(1994)Johnson, Kotz, and Balakrishnan}]{John94}
\bibinfo{author}{N.~L. Johnson}, \bibinfo{author}{S.~Kotz},
  \bibinfo{author}{N.~Balakrishnan}, \bibinfo{title}{Continuous Univariate
  Distributions, Volume 1}, \bibinfo{publisher}{Wiley-Interscience},
  \bibinfo{year}{1994}.
\bibitem[{Hubert and Arabie(1985)}]{HUBAR85}
\bibinfo{author}{L.~Hubert}, \bibinfo{author}{P.~Arabie},
\newblock \bibinfo{title}{Comparing partitions},
\newblock \bibinfo{journal}{Journal of Classification} \bibinfo{volume}{2}
  (\bibinfo{year}{1985}) \bibinfo{pages}{193--218}.
\bibitem[{Calinski and Harabasz(1974)}]{Cali75}
\bibinfo{author}{R.~B. Calinski}, \bibinfo{author}{J.~Harabasz},
\newblock \bibinfo{title}{A dendrite method for cluster analysis},
\newblock \bibinfo{journal}{Communications in Statistics} \bibinfo{volume}{3
  (1)} (\bibinfo{year}{1974}) \bibinfo{pages}{1--27}.

\end{thebibliography}







\end{document}